\documentclass[11pt,reqno]{amsart}
\usepackage{amsmath,amssymb,amsthm,mathrsfs,enumerate,bm,xcolor,multirow,pbox,latexsym}
\usepackage{graphicx,color,framed,tikz}
\usepackage[colorlinks,linkcolor=red,citecolor=blue,urlcolor=blue]{hyperref}
\allowdisplaybreaks[4]
\numberwithin{equation}{section}
\newcommand{\N}{\mathbb{N}}
\newcommand{\R}{\mathbb{R}}

\newcommand{\E}{\mathbb{E}}
\newcommand{\Prob}{\mathbb{P}}
\newcommand{\G}{\mathbb{G}}

\AtBeginDocument{%
	\def\MR#1{}
}

\newcommand{\pnorm}[2]{\lVert#1\rVert_{#2}}
\newcommand{\bigpnorm}[2]{\big\lVert#1\big\rVert_{#2}}
\newcommand{\biggpnorm}[2]{\bigg\lVert#1\bigg\rVert_{#2}}
\newcommand{\abs}[1]{\lvert#1\rvert}
\newcommand{\bigabs}[1]{\big\lvert#1\big\rvert}
\newcommand{\biggabs}[1]{\bigg\lvert#1\bigg\rvert}

\renewcommand{\epsilon}{\varepsilon}

\renewcommand{\d}[1]{\mathrm{d}#1}

\renewcommand{\hat}{\widehat}
\renewcommand{\tilde}{\widetilde}

\newcommand{\beq}{\begin{equation}}
\newcommand{\eeq}{\end{equation}}
\newcommand{\beqa}{\begin{equation} \begin{aligned}}
\newcommand{\eeqa}{\end{aligned} \end{equation}}
\newcommand{\beqas}{\begin{equation*} \begin{aligned}}
\newcommand{\eeqas}{\end{aligned} \end{equation*}}

\newcommand{\bit}{\begin{itemize}}
	\newcommand{\eit}{\end{itemize}}
\newcommand{\bmat}{\begin{bmatrix}}
	\newcommand{\emat}{\end{bmatrix}}

\theoremstyle{definition}\newtheorem{problem}{Problem}[section]
\theoremstyle{definition}\newtheorem{definition}[problem]{Definition}
\theoremstyle{remark}\newtheorem{assumption}{Assumption}

\theoremstyle{remark}\newtheorem{remark}[problem]{Remark}
\theoremstyle{definition}
\theoremstyle{plain}\newtheorem{theorem}[problem]{Theorem}
\theoremstyle{plain}
\theoremstyle{plain}\newtheorem{lemma}[problem]{Lemma}
\theoremstyle{plain}\newtheorem{proposition}[problem]{Proposition}
\theoremstyle{plain}\newtheorem{corollary}[problem]{Corollary}
\theoremstyle{plain}

\begin{document}

\title[Multiplier U-Processes]{Multiplier U-processes: sharp bounds and applications}
\thanks{The research of Q. Han is partially supported by NSF Grant DMS-1916221. }

\author[Q. Han]{Qiyang Han}

\address[Q. Han]{
	Department of Statistics, Rutgers University, Piscataway, NJ 08854, USA.
}
\email{qh85@stat.rutgers.edu}

\date{\today}

\maketitle

\begin{abstract}
The theory for multiplier empirical processes has been one of the central topics in the development of the classical theory of empirical processes, due to its wide applicability to various statistical problems. In this paper, we develop theory and tools for studying multiplier $U$-processes, a natural higher-order generalization of the multiplier empirical processes. To this end, we develop a multiplier inequality that quantifies the moduli of continuity of the multiplier $U$-process in terms of that of the (decoupled) symmetrized $U$-process. The new inequality finds a variety of applications including (i) multiplier and bootstrap central limit theorems for $U$-processes, (ii) general theory for bootstrap $M$-estimators based on $U$-statistics, and (iii) theory for $M$-estimation under general complex sampling designs, again based on $U$-statistics.
\end{abstract}

\section{Introduction}
\subsection{Overview}
Let $X_1,\ldots,X_n$ be i.i.d. random variables with law $P$ on $(\mathcal{X},\mathcal{A})$, and $\xi_1,\ldots,\xi_n$ be random variables independent of $X_1,\ldots,X_n$. Multiplier empirical processes of the form
\begin{align}\label{def:multiplier_empirical_process}
f \mapsto \sum_{i=1}^n \xi_i f(X_i),
\end{align}
where $f \in \mathcal{F}$ for some function class $\mathcal{F}$, have a long history in the development of the classical empirical process theory \cite{van1996weak,ledoux2013probability}. See also \cite{mendelson2016upper,mendelson2016multiplier} for some recent developments. Apart from being of theoretical interest in its own right, the multiplier empirical process has also found numerous important applications in the statistical theory. Here is a partial list:
\begin{enumerate}
	\item[(P1)] (\emph{Bootstrap theory}). The seminal paper \cite{gine1990bootstrapping} gives sufficient and necessary characterizations for the bootstrap central limit theorem to hold uniformly over a class of functions $\mathcal{F}$. The key idea there is to view the bootstrap empirical process as certain (conditional) multiplier empirical process. This idea is further exploited in \cite{praestgaard1993exchangeably} to general bootstrap schemes with exchangeable weights.
	\item[(P2)] (\emph{Estimation theory}). The bootstrap (multiplier empirical) theory in (P1) can be combined with $M$- or $Z$-estimation theory to study asymptotic properties of bootstrap finite-dimensional parameters: the paper \cite{wellner1996bootstrapping} studied bootstrap $Z$-estimators; the paper \cite{cheng2010bootstrap} further studied bootstrap $M$-estimators in a semiparametric setting. In an infinite-dimensional setting, the multiplier empirical process naturally arises in the theory for regression estimators, where the multipliers play the role of the errors in the regression model, cf. \cite{han2017sharp}.

	\item[(P3)] (\emph{Complex sampling}). \cite{breslow2007weighted} pioneered the study of empirical process theory under two-phase stratified sampling by  using the exchangeably weighted bootstrap empirical process theory developed in \cite{praestgaard1993exchangeably}. Therefore the crux of problem rests in suitable form of the multiplier empirical process theory. 	
\end{enumerate}

As a natural analogue of the empirical process in a higher-order setting, $U$-process (of order $m$) of the form
\begin{align}\label{def:U_process}
f \mapsto \sum_{1\leq i_1<\ldots<i_m\leq n} f(X_{i_1},\ldots,X_{i_m})
\end{align}
received considerable attention during the late 1980s and early 1990s due to its wide applications to the statistical theory, see e.g. \cite{arcones1993limit,arcones1994Uprocesses,arcones1994central,arcones1994estimators,nolan1987Uprocesses,nolan1988functional}. On the other hand, despite notable progress of the theory for the multiplier empirical process (\ref{def:multiplier_empirical_process}) and its applications thereof, corresponding theory for multiplier $U$-processes of the form
\begin{align}\label{def:multiplier_U_process}
f \mapsto \sum_{1\leq i_1<\ldots<i_m\leq n} \xi_{i_1}\cdots \xi_{i_m}f(X_{i_1},\ldots,X_{i_m}),
\end{align}
has been lacking. Not surprisingly, the lack of such a theory has hindered further theoretical understanding for various statistical problems (in particular the above (P1)-(P3)) that involve $U$-statistics. One of the goals of this paper is to fill in this lack of understanding, by developing further theory and tools for understanding multiplier $U$-processes (\ref{def:multiplier_U_process}), along with a particular eye on applications to the aforementioned statistical problems.

It has now been clear from the author's previous work \cite{han2017sharp,han2017robustness,han2018complex} that the key step in getting a precise understanding of the behavior of the multiplier empirical process (\ref{def:multiplier_empirical_process}) is a strong form of the so-called `multiplier inequality' that quantifies the moduli of continuity of the multiplier empirical process in terms of that of the empirical process itself, or its symmetrized equivalent
\begin{align}
f \mapsto \sum_{i=1}^n \epsilon_i f(X_i),
\end{align}
in a non-asymptotic setting. Here $\epsilon_i$'s are i.i.d. Rademacher random variables (i.e. $\Prob(\epsilon_i=\pm 1)=1/2$) independent of $X_1,\ldots,X_n$. Indeed, an improved version of the multiplier inequality is proved in \cite{han2017sharp} that gives sharp non-asymptotic bounds for the moduli of the multiplier empirical process. The benefits of such a sharp multiplier inequality are exploited in a few different problems, including (i) convergence rates of least squares regression estimators in a heavy-tailed regression setting under various models \cite{han2017sharp,han2017robustness}; (ii) empirical process theory under general complex sampling designs \cite{han2018complex}, and more technically, (iii) theory for lower bounds of empirical processes through Gaussianization \cite{han2019global}.

This is the program we will continue in this paper, but now in the setting of multiplier $U$-process (\ref{def:multiplier_U_process}). We prove in Theorem \ref{thm:multiplier_ineq} a sharp multiplier inequality used to study the moduli of continuity of the multiplier $U$-process (\ref{def:multiplier_U_process}), in terms of that of the (decoupled) symmetrized $U$-process
\begin{align}\label{def:symmetrized_U_process}
f \mapsto \sum_{1\leq i_1<\ldots<i_m\leq n} \epsilon_{i_1}^{(1)}\cdots \epsilon_{i_m}^{(m)} f(X_{i_1}^{(1)},\ldots,X_{i_m}^{(m)}),
\end{align}
an object that has been well-studied throughout a series of ground-breaking works in the 1990s, cf. \cite{arcones1993limit,arcones1994Uprocesses,de2012decoupling}.

With the help of the multiplier inequality for the multiplier $U$-process (\ref{def:multiplier_U_process}), we further study the statistical problems in the directions (P1)-(P3) mentioned above, in which various forms of $U$-statistics are involved. More specifically:
\begin{enumerate}
	\item[(Q1)] We prove a multiplier central limit theorem and a bootstrap central limit theorem with general exchangeable weights for $U$-processes, in analogy to the duality between the multiplier central limit theorem for empirical processes developed in \cite{gine1984some,gine1986lectures} and the exchangeably weighted bootstrap theory for empirical processes developed in \cite{praestgaard1993exchangeably}. These uniform central limit theorems are valid under the same conditions for the usual uniform CLTs for $U$-processes as developed in \cite{arcones1993limit,de2012decoupling}, and similar conditions on the exchangeable weights as in \cite{praestgaard1993exchangeably}.
	\item[(Q2)] We develop a general theory for bootstrap $M$-estimators based on $U$-statistics, continuing the line of works pioneered by \cite{wellner1996bootstrapping} in the context of $Z$-estimation based on the empirical measure, and by \cite{cheng2010bootstrap} in the context of $M$-estimation in a semi-parametric setting. See also \cite{bose2001generalised,bose2003generalized,chatterjee2005generalized}. The bootstrap theory is valid under essentially the same conditions as the master asymptotic normality theorem as studied in \cite{arcones1994estimators,de2012decoupling}, and therefore validates the use of bootstrap $M$-estimators based on $U$-statistics, at least asymptotically.
	\item[(Q3)] We develop a general $M$-estimation theory based on $U$-statistics under general complex sampling designs. Our theory generalizes the work of \cite{arcones1994estimators,de2012decoupling} concerning finite-dimensional $M$-estimators based on $U$-statistics, and the work of \cite{clemencon2008ranking} concerning infinite-dimensional $M$-estimators based on $U$-statistics, by going beyond the standard setting of i.i.d. sampling design. The theory here can also be viewed as an extension of the author's previous work \cite{han2018complex}, in which a general empirical process theory for various sampling designs is developed with the help of the improved multiplier inequality for empirical processes proved in \cite{han2017sharp}.
\end{enumerate}

The rest of the paper is organized as follows. We develop the multiplier inequality for $U$-processes in Section \ref{section:multiplier_ineq}. Sections \ref{section:multiplier_CLT_bootstrap}-\ref{section:complex_sampling_M} are devoted to applications (Q1)-(Q3). Most detailed proofs are collected in Sections \ref{section:proof_1}-\ref{section:proof_auxiliary_results}.

\subsection{Notation}\label{section:notation}
For a real-valued random variable $\xi$ and $1\leq p<\infty$, let $\pnorm{\xi}{p} \equiv \big(\E\abs{\xi}^p\big)^{1/p} $ denote the ordinary $p$-norm. The $L_{p,1}$ `norm' for a random variable $\xi$ is defined by 
\begin{align*}
\pnorm{\xi}{p,1}\equiv\int_0^\infty {\Prob(\abs{\xi}>t)}^{1/p}\ \d{t}.
\end{align*}
Strictly speaking $\pnorm{\cdot}{p,1}$ is not a norm, but there exists a norm equivalent to $\pnorm{\cdot}{p,1}$ (cf.  \cite[Theorem 3.21]{stein1971introduction}). Let $L_{p,1}$ be the space of random variables $\xi$'s with $\pnorm{\xi}{p,1}<\infty$. It is well known that $L_{p+\epsilon}\subset L_{p,1}\subset L_{p}$ holds for any underlying probability measure, and hence a finite $L_{p,1}$ condition requires slightly more than a $p$-th moment, but no more than any $p+\epsilon$ moment, see \cite[Chapter 10]{ledoux2013probability}.

For a real-valued measurable function $f$ defined on $(\mathcal{X},\mathcal{A},P)$, $\pnorm{f}{L_p(P)}\equiv \pnorm{f}{P,p}\equiv \big(P\abs{f}^p)^{1/p}$ denotes the usual $L_p$-norm under $P$, and $\pnorm{f}{\infty}\equiv \sup_{x \in \mathcal{X}} \abs{f(x)}$. $f$ is said to be $P$-centered if $Pf=0$, and $\mathcal{F}$ is $P$-centered if all $f \in \mathcal{F}$ are $P$-centered. To avoid unnecessary measurability digressions, we will assume that $\mathcal{F}$ is countable throughout the article. As usual, for any $\phi:\mathcal{F}\to \R$, we write $\pnorm{\phi(f)}{\mathcal{F}}$ for $ \sup_{f \in \mathcal{F}} \abs{\phi(f)}$.

Let $(\mathcal{F},\pnorm{\cdot}{})$ be a subset of the normed space of real functions $f:\mathcal{X}\to \R$. For $\epsilon>0$ let $\mathcal{N}(\epsilon,\mathcal{F},\pnorm{\cdot}{})$ be the $\epsilon$-covering number of $\mathcal{F}$; see \cite[page 83]{van1996weak} for more details.

Throughout the article $\epsilon_1,\ldots,\epsilon_n$ will be i.i.d. Rademacher random variables independent of all other random variables. $C_{x}$ will denote a generic constant that depends only on $x$, whose numeric value may change from line to line unless otherwise specified. $a\lesssim_{x} b$ and $a\gtrsim_x b$ mean $a\leq C_x b$ and $a\geq C_x b$ respectively, and $a\asymp_x b$ means $a\lesssim_{x} b$ and $a\gtrsim_x b$ [$a\lesssim b$ means $a\leq Cb$ for some absolute constant $C$]. For two real numbers $a,b$, $a\vee b\equiv \max\{a,b\}$ and $a\wedge b\equiv\min\{a,b\}$. $\mathcal{O}_{\mathbf{P}}$ and $\mathfrak{o}_{\mathbf{P}}$ denote the usual big and small O notation in probability.

\section{Multiplier inequality for $U$-processes}\label{section:multiplier_ineq}

Recall that a kernel $f:\mathcal{X}^m \to \R$ is (permutation) symmetric if and only if $f(x_1,\ldots,x_m)= f(x_{\pi(1)},\ldots,x_{\pi(m)})$ for any permutation $\pi$ of $\{1,\ldots,m\}$. Further recall the notion of degenerate kernels (cf. \cite[Definition 3.5.1]{de2012decoupling}) as follows.
\begin{definition}
A symmetric and $P^m$-integrable kernel $f:\mathcal{X}^m \to \R$ is $P$-degenerate of order $r-1$, notationally $f \in L_2^r(P^m)$, if and only if
\begin{align*}
\int f(x_1,\ldots,x_m)\ \d{P^{m-r+1}}(x_r,\ldots,x_m)= \int f\ \d{P^m}
\end{align*}
holds for any $x_1,\ldots,x_{r-1} \in \mathcal{X}$, and 
\begin{align*}
(x_1,\ldots,x_r)\mapsto \int f(x_1,\ldots,x_m)\ \d{P}^{m-r}(x_{r+1},\ldots,x_m) 
\end{align*}
is not a constant function.  If $f$ is furthermore $P^m$-centered, i.e. $P^m f=0$, we write $f \in L_2^{c,r}(P^m)$. For notational simplicity, we usually write $L_2^{c,m}(P^m)=L_2^{c,m}(P)$.
\end{definition}

Any $U$-statistic with a symmetric kernel $f$
\begin{align}\label{eqn:U_stat}
U_n^{(m)} (f) \equiv \frac{1}{\binom{n}{m}}\sum_{1\leq i_1<\ldots<i_m\leq n} f(X_{i_1},\ldots,X_{i_m})
\end{align}
can be decomposed into the sum of $U$-statistics with degenerate kernels:
\begin{align}\label{eqn:hoeffding_decomposition}
U_n^{(m)}(f) = \sum_{k=0}^m \binom{m}{k} U_n^{(k)}(\pi_k f).
\end{align}
Here $\pi_k f(x_1,\ldots,x_k)\equiv (\delta_{x_1}-P)\times\ldots\times (\delta_{x_k}-P)\times P^{m-k} f$ is $P$-degenerate of order $k-1$. The equation (\ref{eqn:hoeffding_decomposition}) is also known as Hoeffding decomposition. The limit behavior of (\ref{eqn:U_stat}) then typically reduces to the study of the leading non-trivial term ($k\geq 1$) in (\ref{eqn:hoeffding_decomposition}), so below we proceed without loss of generality by assuming that the kernels $f$ are $P$-degenerate of order $m-1$ unless otherwise specified.

The main result of this section is the following multiplier inequality for $U$-processes with degenerate kernels. 

\begin{theorem}\label{thm:multiplier_ineq}
Let $X_1,\ldots,X_n$ be i.i.d. random variables with law $P$ on $(\mathcal{X},\mathcal{A})$, and $\mathcal{F}$ be a class of measurable real-valued functions defined on $(\mathcal{X}^m,\mathcal{A}^m)$ such that $\mathcal{F}$ is $P$-centered and $P$-degenerate of order $m-1$. Let $(\xi_1,\ldots,\xi_n)$ be a random vector independent of $(X_1,\ldots,X_n)$. Suppose that there exists some measurable function $\psi_n: \R^m_{\geq 0}\to \R_{\geq 0}$ such that the expected supremum of the decoupled \footnote{`Here `decoupled' refers to fact that $\{X_i^{(k)}\}, k \in \N$ are independent copies of $\{X_i\}$, and $\{\epsilon_i^{(k)}\}, k\in \N$ are independent copies of the Rademacher sequence $\{\epsilon_i\}$.} $U$-processes satisfies
\begin{align*}
\E \biggpnorm{\sum_{\substack{1\leq i_k\leq \ell_k, 1\leq k\leq m}} \epsilon_{i_1}^{(1)}\cdots\epsilon_{i_m}^{(m)} f(X_{i_1}^{(1)},\ldots,X_{i_m}^{(m)})}{\mathcal{F}}\leq \psi_n(\ell_1,\ldots,\ell_m)
\end{align*}
for all $1\leq \ell_1,\ldots,\ell_m\leq n$. Then
\begin{align*}
&\E \biggpnorm{\sum_{1\leq i_1,\ldots,i_m\leq n} \xi_{i_1}\cdots\xi_{i_m} f(X_{i_1},\ldots,X_{i_m})}{\mathcal{F}} \\
&\leq K_m \int_{\R_{\geq 0}^m}    \E \psi_n\bigg(\sum_{i=1}^n \bm{1}_{\abs{\xi_i}>t_1},\ldots,\sum_{i=1}^n \bm{1}_{\abs{\xi_i}>t_m}\bigg)\ \d{t_1}\cdots\d{t_m}.
\end{align*}
Furthermore, if there exists a concave and non-decreasing function $\bar{\psi}_n:\R\to \R$ such that $\psi_n(\ell_1,\ldots,\ell_m) = \bar{\psi}_n \big(\prod_{k=1}^m \ell_k\big)$, then
\begin{align*}
&\E \biggpnorm{\sum_{1\leq i_1,\ldots,i_m\leq n} \xi_{i_1}\cdots\xi_{i_m} f(X_{i_1},\ldots,X_{i_m})}{\mathcal{F}} \\
&\leq K_m \int_{\R_{\geq 0}^m }   \bar{\psi}_n \bigg(  \sum_{1\leq i_1,\ldots,i_m\leq n}  \prod_{k=1}^m \Prob\big(\abs{\xi_{i_k}}>t_k\big)^{1/m}\bigg)\ \d{t_1}\cdots\d{t_m}.
\end{align*}
Here $K_m>0$ is a constant depending on $m$ only, and can be taken as $K_m=2^{2m}\prod_{k=2}^m (k^k-1)$ for $m\geq 2$.
\end{theorem}

As an immediate consequence of Theorem \ref{thm:multiplier_ineq}, we have the following:
\begin{corollary}\label{cor:multiplier_ineq}
Suppose that the conditions on $(X_1,\ldots,X_n)$ and $(\xi_1,\ldots,\xi_n)$ in Theorem \ref{thm:multiplier_ineq} hold, and that $\xi_i$'s have the same marginal distributions. If there exist some $\gamma>1$ and $\kappa_0>0$ such that 
\begin{align}\label{ineq:cor_multiplier_1}
\E \biggpnorm{\sum_{\substack{1\leq i_k\leq \ell_k, 1\leq k\leq m}} \epsilon_{i_1}^{(1)}\cdots\epsilon_{i_m}^{(m)} f(X_{i_1}^{(1)},\ldots,X_{i_m}^{(m)})}{\mathcal{F}}\leq  \kappa_0 \prod_{k=1}^m \ell_k^{1/\gamma}
\end{align}
holds for all $1\leq \ell_1,\ldots,\ell_m\leq n$, then 
\begin{align}\label{ineq:cor_multiplier_2}
\E \biggpnorm{\sum_{1\leq i_1,\ldots,i_m\leq n} \xi_{i_1}\cdots\xi_{i_m} f(X_{i_1},\ldots,X_{i_m})}{\mathcal{F}} \leq K_m \kappa_0 \pnorm{\xi_1}{m\gamma,1}^m\cdot n^{ m/\gamma}.
\end{align}
\end{corollary}
\begin{proof}
Let $\psi_n(\ell_1,\ldots,\ell_m)\equiv \kappa_0 \big(\prod_{k=1}^m \ell_k\big)^{1/\gamma}\equiv \bar{\psi}_n\big(\prod_{k=1}^m \ell_k\big)$ where $\bar{\psi}_n(t)\equiv \kappa_0\cdot t^{1/\gamma}$. By Theorem \ref{thm:multiplier_ineq}, the LHS of the above display can be bounded by
\begin{align*}
& K_m \kappa_0 \int_{\R_{\geq 0}^m }   \bigg(  \sum_{1\leq i_1,\ldots,i_m\leq n}  \prod_{k=1}^m \Prob\big(\abs{\xi_1}>t_k\big)^{1/m}\bigg)^{1/\gamma}\ \d{t_1}\cdots\d{t_m} \\
&= K_m\kappa_0\cdot n^{m/\gamma}  \prod_{k=1}^m \int_0^\infty   \Prob(\abs{\xi_1}>t_k)^{1/m\gamma}\ \d{t_k}= K_m \kappa_0 \pnorm{\xi_1}{m\gamma,1}^m\cdot n^{ m/\gamma} ,
\end{align*}
as desired.
\end{proof}

The above corollary shows that the multiplier $U$-process (\ref{ineq:cor_multiplier_2}) enjoys the same size $n^{m/\gamma}$ as the Rademacher randomized $U$-process (\ref{ineq:cor_multiplier_1}), as long as the multipliers $\xi_i$'s satisfy the moment condition $\pnorm{\xi_1}{m\gamma}<\infty$.  Whether this moment condition is necessary remains open for general $m\geq 2$. For $m=1$, this moment condition cannot be substantially improved, see \cite[Theorem 2]{han2017sharp}.

The next proposition shows certain sharpness of Corollary \ref{cor:multiplier_ineq} in terms of the size of the multiplier $U$-process. In particular, we prove that there exists $\mathcal{F}$ verifying the condition (\ref{ineq:cor_multiplier_1}), while the inequality (\ref{ineq:cor_multiplier_2}) cannot be further improved for i.i.d. centered multipliers $\xi_i$'s with sufficient moments. 

\begin{proposition}\label{prop:lower_bound_multiplier_ineq}
Suppose that $\mathcal{X}\equiv  [0,1]$ and $P$ is the uniform probability measure on $\mathcal{X}$. Fix $\gamma>2$. Then there exists some $\mathcal{F}$ defined on $\mathcal{X}$ for which
	\begin{align*}
	\E \biggpnorm{\sum_{\substack{1\leq i_k\leq \ell_k, 1\leq k\leq m}} \epsilon_{i_1}^{(1)}\cdots\epsilon_{i_m}^{(m)} f(X_{i_1}^{(1)},\ldots,X_{i_m}^{(m)})}{\mathcal{F}}\leq  \kappa_0 \prod_{k=1}^m \ell_k^{1/\gamma}
	\end{align*} 
	holds for all $1\leq \ell_1,\ldots,\ell_m\leq n$, such that for any centered i.i.d. random variables $\xi_1,\ldots,\xi_n$ with $0<\pnorm{\xi_1}{1}<\infty$, 
	\begin{align*}
	\E \biggpnorm{\sum_{1\leq i_1,\ldots,i_m\leq n} \xi_{i_1}\cdots\xi_{i_m} f(X_{i_1},\ldots,X_{i_m})}{\mathcal{F}} \geq \kappa_{m,\gamma,\xi} \cdot n^{m/\gamma}.
	\end{align*}
	Here the constant $\kappa_{m,\gamma,\xi}$ only depends on $m,\gamma$ and the law of $\xi_1$.
\end{proposition}

\begin{remark}
In the special case of $m=1$, the multiplier inequality for $U$-processes in Theorem \ref{thm:multiplier_ineq} reduces to (a special case of) the improved multiplier inequality for empirical processes proved in \cite[Theorem 1]{han2017sharp}. The reader is referred to \cite[Section 2.3]{han2017sharp} for a detailed comparison of the improvement in this case over the classical multiplier inequality (cf. \cite[Lemma 2.9.1]{van1996weak}).
\end{remark}

In the applications in the next section, the following result will be useful in verifying asymptotic equicontinuity of the multiplier $U$-processes.

\begin{corollary}\label{cor:asymptotic_equicont}
Consider the setup of Theorem \ref{thm:multiplier_ineq}. Let $\{\mathcal{F}_{(\ell_1,\ldots,\ell_m),n}: 1\leq \ell_1,\ldots,\ell_m\leq n, n \in \N\}$ be  function classes such that $\mathcal{F}_{(\ell_1,\ldots,\ell_m),n}\supset \mathcal{F}_{(n,\ldots,n),n}$ for all $1\leq \ell_1,\ldots,\ell_m\leq n$. Suppose that $\xi_i$'s have the same marginal distributions with $\pnorm{\xi_1}{2m,1}<\infty$. Suppose that there exists some bounded measurable function $a: \R^m_{\geq 0}\to \R_{\geq 0}$ with $a(\ell_1,\ldots,\ell_m)\to 0$ as $\ell_1\wedge \ldots \wedge \ell_m \to \infty$, such that the expected supremum of the decoupled $U$-processes  satisfies
\begin{align*}
&\E \biggpnorm{\sum_{\substack{1\leq i_k\leq \ell_k, 1\leq k\leq m}} \epsilon_{i_1}^{(1)}\cdots\epsilon_{i_m}^{(m)} f(X_{i_1}^{(1)},\ldots,X_{i_m}^{(m)})}{\mathcal{F}_{(\ell_1,\ldots,\ell_m),n} }\\
&\qquad \leq a(\ell_1,\ldots,\ell_m) \bigg(\prod_{k=1}^m \ell_k\bigg)^{1/2}
\end{align*}
for all $1\leq \ell_1,\ldots,\ell_m\leq n$. Then
\begin{align*}
n^{-m/2} \E \biggpnorm{\sum_{1\leq i_1,\ldots,i_m\leq n} \xi_{i_1}\cdots\xi_{i_m} f(X_{i_1},\ldots,X_{i_m})}{\mathcal{F}_{(n,\ldots,n),n}} \to 0,\quad n \to \infty.
\end{align*}
\end{corollary}

\section{Multiplier central limit theorem and the bootstrap}\label{section:multiplier_CLT_bootstrap}

In this section, we will apply the multiplier inequality in Theorem \ref{thm:multiplier_ineq} to establish a multiplier central limit theorem and an exchangeably weighted bootstrap central limit theorem for $U$-processes. As already mentioned in the introduction, the duality between these two limit theorems is akin to the development from the empirical process theory side: a multiplier central limit theorem for empirical processes is established in \cite{gine1984some,gine1986lectures}, which serves as a cornerstone for the bootstrap central limit theorems in \cite{gine1990bootstrapping,praestgaard1993exchangeably}.

Below we review some basic facts for the central limit theorems for degenerate $U$-statistics. The materials below come from \cite[Section 4.2 ]{de2012decoupling}; the reader is referred therein for a more detailed description. Let $K_{P}$ be the Gaussian chaos process defined on $\R\oplus L_2^{c,\N}(P)\equiv \R\oplus \big(\oplus_{m=1}^\infty L_2^{c,m}(P)\big)$ as follows\footnote{ $\oplus$ is the orthogonal sum in $L_2(\mathcal{X}^{\infty},\mathcal{A}^{\infty},P^{\infty})$.}. Let $h^\psi_m(x_1,\ldots,x_m)=\prod_{\ell=1}^m \psi(x_\ell)$ for $\psi \in L_2^{c,1}(P)$. Then the linear span of $\{h^\psi_m: \psi \in L_2^{c,1}(P)\}$ is dense in the space of $L_2^{c,m}(P)$ with respect to $L_2(P^m)$. Hence we may define
\begin{align}
K_{P}(h^\psi_m)\equiv (m!)^{1/2} R_m\big(G_P(\psi),\E\psi^2,0,\ldots,0\big),
\end{align}
and extend this map linearly and continuously, with $K_P(1)\equiv 1$, on $\R\oplus L_2^{c,\N}(P)$. Here $G_P$ is the isonormal Gaussian process on $L_2^{c,1}(P)$ with covariance structure $\E G_P (f) G_P(g)= P(fg) (f,g \in L_2^{c,1}(P))$, and $R_m$ is the polynomial of degree $m$ given by the Newton's identity (cf.  \cite[pp. 175]{de2012decoupling}):
\begin{align}\label{eqn:Newton_identity}
\sum_{1\leq i_1<\ldots<i_m\leq n} t_{i_1}\cdots t_{i_m} = R_m\bigg(\sum_{i=1}^n t_i,\sum_{i=1}^n t_i^2,\ldots,\sum_{i=1}^n t_i^m\bigg).
\end{align}
With these notations, if $f_\ell \in L_2^{c,m_\ell}(P) (1\leq \ell \leq k)$, then the following central limit theorem holds (cf. \cite[Theorem 4.2.4]{de2012decoupling}):
\begin{align*}
\bigg(\binom{n}{m_1}^{1/2} U_n^{(m_1)}(f_1),\ldots \binom{n}{m_k}^{1/2} U_n^{(m_k)}(f_k) \bigg) \rightsquigarrow_d \big(K_P(f_1),\ldots,K_P(f_k)\big).
\end{align*}
It is also well-known that if a class of measurable functions $\mathcal{F}$ satisfies good entropy conditions, then a central limit theorem in $\ell^\infty(\mathcal{F})$ holds (cf. \cite[Theorem 5.3.7]{de2012decoupling}): for any $1\leq k\leq m$,
\begin{align*}
\bigg\{ \binom{n}{k}^{1/2} U_n^{(k)}(\pi_k f): f \in \mathcal{F}\bigg\}\rightsquigarrow_d \big\{K_P(\pi_k f): f \in \mathcal{F}\big\}\textrm{  in }\ell^\infty(\mathcal{F}),
\end{align*}
or equivalently,
\begin{align*}
\sup_{\psi \in \mathrm{BL}(\ell^\infty(\mathcal{F})) } \biggabs{\E^\ast \psi\bigg(\binom{n}{k}^{1/2}U_n^{(k)}(\pi_k)\bigg)-\E \psi(K_P(\pi_k))}\to 0,
\end{align*}
where $\E^\ast$ is the outer expectation (cf. \cite[Section 1.2]{van1996weak})). Now we consider the \emph{multiplier $U$-process}: for any $f \in \mathcal{F}$ and $\xi_i$'s, define
\begin{align}
U_{n,\xi}^{(m)}(f)\equiv \frac{1}{\binom{n}{m}}\sum_{1\leq i_1<\ldots<i_m\leq n} \xi_{i_1}\cdots\xi_{i_m}f(X_{i_1},\ldots,X_{i_m}).
\end{align}
Our first result in this section establishes a multiplier central limit theorem for $U$-processes.
\begin{theorem}\label{thm:multiplier_CLT}
Let $\{\xi_i\}_{i=1}^\infty$ be i.i.d. centered random variables with variance $1$ and $\pnorm{\xi_1}{2m,1}<\infty$. Let $\mathcal{F}\subset L_{2}^{c,m}(P)$ admit a $P^m$-square integrable envelope $F$ such that
\begin{align*}
\int_0^{1} \big(\sup_Q \log\mathcal{N}\big(\epsilon \pnorm{F}{L_2(Q)},\mathcal{F}, L_2(Q)\big)\big)^{m/2}\ \d{\epsilon}<\infty,
\end{align*}
where the supremum is taken over all discrete probability measures. Then
\begin{align*}
\sup_{\psi \in \mathrm{BL}(\ell^\infty(\mathcal{F}))} \biggabs{\E^\ast \psi\bigg(\binom{n}{m}^{1/2}U_{n,\xi}^{(m)}\bigg)-\E \psi(K_P)}\to 0.
\end{align*}
\end{theorem}

Note that the entropy condition required in Theorem \ref{thm:multiplier_CLT} is exactly the same for the uniform central limit theorem of $U$-processes (cf. \cite{de2012decoupling,arcones1993limit}). Furthermore, the moment condition for the multipliers is a finite $L_{2m,1}$ moment, which agrees with the multiplier central limit theorem for empirical processes when $m=1$, cf. \cite{ledoux1986conditions,ledoux2013probability,van1996weak}.

One natural `application'  for the multiplier central limit theorem in Theorem \ref{thm:multiplier_CLT} is to suggest how to proceed with the formulation of the bootstrap for $U$-processes with general weights. First, let us state some assumptions on the weights.

\begin{assumption}\label{assumption:weights}
	Assume the following conditions on the weight $(\xi_1,\ldots,\xi_n)\equiv (\xi_{n1},\ldots,\xi_{nn})$ defined on $(\mathcal{W},\mathcal{A}_\xi,P_\xi)$:
	\begin{enumerate}
		\item[(W1)] $(\xi_1,\ldots,\xi_n)$ are exchangeable\footnote{In other words, $(\xi_1,\ldots,\xi_n)=_d(\xi_{\pi(1)},\ldots, \xi_{\pi(n)} )$ for any permutation $\pi$ of $\{1,\ldots,n\}.$}, non-negative and $\sum_{i=1}^n\xi_i=n$.
		\item[(W2)] $\sup_n \pnorm{\xi_1}{2m,1}<\infty$, $n^{-1}\max_{1\leq i\leq n}(\xi_i-1)^2\to_{P_\xi} 0$ and there exists $c>0$ such that
		\begin{align*}
		\frac{1}{n}\sum_{i=1}^n(\xi_i-1)^2\to_{P_\xi} c^2.
		\end{align*}
	\end{enumerate}
\end{assumption}
These assumptions are familiar in the context of exchangeably weighted bootstrap limit theory for empirical processes developed in \cite{praestgaard1993exchangeably}. For instance, by taking $(\xi_1,\ldots,\xi_n)\equiv \mathrm{Multinomial}(n,(1/n),\ldots,(1/n))$ we obtain Efron's bootstrap; by taking $\xi_i \equiv Y_i/\bar{Y}$ where $Y_i\sim_{\mathrm{i.i.d.}} \mathrm{exp}(1), \bar{Y} = n^{-1}\sum_{i=1}^n Y_i$ we obtain the Bayesian bootstrap. We refer the reader to \cite{praestgaard1993exchangeably} for a detailed account for various bootstrap proposals.
	
The condition $n^{-1}\max_{1\leq i\leq n}(\xi_i-1)^2\to_{P_\xi} 0$ is automatically satisfied by the moment assumption $\sup_n \pnorm{\xi_1}{2m,1}<\infty$ when $m\geq 2$. We include this condition here to match the same conditions as studied for $m=1$ in \cite{praestgaard1993exchangeably}. 

For any $f \in L_2^{c,m}(P)$, let
\begin{align}
\tilde{U}_{n,\xi}^{(m)}(f)\equiv \frac{1}{\binom{n}{m}} \sum_{1\leq i_1<\ldots<i_m\leq n} (\xi_{i_1}-1)\cdots(\xi_{i_m}-1) f(X_{i_1},\ldots,X_{i_m}).
\end{align}
\cite{huskova1993consistency} considered the special case $m=2$ and derived asymptotic distribution for a single function $f$. Below we will prove a bootstrap uniform central limit theorem.

\begin{theorem}\label{thm:bootstrap_CLT}
Suppose Assumption \ref{assumption:weights} holds. Let $\mathcal{F}\subset L_{2}^{c,m}(P)$ admit a $P^m$-square integrable envelope $F$ such that
\begin{align*}
\int_0^{1} \big(\sup_Q \log\mathcal{N}\big(\epsilon \pnorm{F}{L_2(Q)},\mathcal{F}, L_2(Q)\big)\big)^{m/2}\ \d{\epsilon}<\infty,
\end{align*}
where the supremum is taken over all discrete probability measures. Then
\begin{align*}
\sup_{\psi \in \mathrm{BL}(\ell^\infty(\mathcal{F}))} \biggabs{\E_\xi \psi\bigg(\binom{n}{m}^{1/2}\tilde{U}_{n,\xi}^{(m)}\bigg)-\E \psi(c\cdot K_P)}\to_{P_X^\ast} 0,
\end{align*}
where $c$ is the constant in (W2), and the convergence in probability $\to_{P_X^\ast}$ is with respect to the outer probability of $P^\infty$ defined on $(\mathcal{X}^\infty, \mathcal{A}^\infty)$.
\end{theorem}
Theorem \ref{thm:bootstrap_CLT} extends the exchangeably weighted bootstrap central limit theorem for the empirical process studied in \cite{praestgaard1993exchangeably} to the context of $U$-processes. To the best knowledge of the author, there is very limited understanding for bootstrap central limit theorems for degenerate $U$-processes. The paper \cite{arcones1994Uprocesses} considered Efron's bootstrap and proved bootstrap CLTs by a straightforward conditioning argument along with the VC-type assumption that gives a uniform control for the empirical measure. The paper \cite{zhang2001bayesian} considered Bayesian bootstrap, but his results are confined to the non-degenerate case. Our Theorem \ref{thm:bootstrap_CLT} holds under the same condition for the CLT for degenerate $U$-processes, and under general bootstrap schemes.

\section{Bootstrapping $M$-estimators}\label{section:bootstrap_M}

In this section, we will investigate the bootstrap theory under the $M$-estimation framework based on $U$-statistics. Let $\Theta \subset \R^d$ index a class of symmetric kernels $\mathcal{F}\equiv \{f_\theta: \mathcal{X}^m\to \R, \theta \in \Theta\}$. Let $\theta_0$ be the unique maximizer of $\theta\mapsto P^m f_\theta$, and an estimator of $\theta_0$ based on $(X_1,\ldots,X_n)$ is given by maximizing a $U$-statistic
\begin{align}
\hat{\theta}_n\in \arg\max_{\theta \in \Theta} U_n^{(m)}(f_\theta) = \arg\max_{\theta \in \Theta} \sum_{i_1\neq \ldots\neq i_m} f_\theta(X_{i_1},\ldots,X_{i_m}).
\end{align}
In typical applications, $\mathcal{F}$ contains non-degenerate (and non-negative) kernels and hence under regularity conditions $\sqrt{n}(\hat{\theta}_n-\theta_0)$ is asymptotically normal, the variance of which depends on the unknown distribution $P$. For bootstrap weights $(\xi_1,\ldots,\xi_n)$ defined on $(\mathcal{W},\mathcal{A}_\xi,P_\xi)$, consider the following bootstrap estimate
\begin{align}
\theta^\ast_n \in \arg\max_{\theta \in \Theta} \sum_{i_1\neq \ldots\neq i_m} \xi_{i_1}\cdots\xi_{i_m} f_\theta(X_{i_1},\ldots,X_{i_m}).
\end{align}
We will be naturally interested in the asymptotic behavior of $\sqrt{n}(\theta^\ast_n-\hat{\theta}_n)$ conditional on the observed data $\{X_i\}$.

Before formally stating our results, we need the following notions concerning bootstrap in probability statements. 

\begin{definition}
Let $\{\Delta_n\}_{n=1}^\infty$ be a sequence of random variables defined on $(\Omega,\mathcal{B},\Prob)=(\mathcal{X}^\infty,\mathcal{A}^\infty, P^\infty)\times(\mathcal{W}, \mathcal{A}_\xi,P_\xi)$.
\begin{enumerate}
\item We say that $\Delta_n\equiv \mathfrak{o}_{P_\xi}(1)$ in $P_{X}$-probability if and only if for any $\epsilon>0$, $\Prob_{W|X}(\Delta_n>\epsilon)=\mathfrak{o}_{P_{X}} (1)$.
\item We say that $\Delta_n\equiv \mathcal{O}_{P_\xi}(1)$ in $P_{X}$-probability if and only if for any $L_n\to \infty$, $\Prob_{W|X}(\Delta_n>L_n)=\mathfrak{o}_{P_X}(1)$.
\end{enumerate}
\end{definition}

The main result of this section is the following theorem.

\begin{theorem}\label{thm:bootstrap_M}
Suppose that the bootstrap weights $(\xi_1,\ldots,\xi_n)$ satisfy Assumption \ref{assumption:weights}, and the following conditions hold.
\begin{enumerate}
	\item[(M1)] The map $\theta \mapsto D(f_\theta)\equiv P^m f_\theta$ has a unique maximizer at $\theta=\theta_0$ and there exists some positive definite matrix $V$ such that for $\theta \in \Theta$ close enough to $\theta_0$,
	\begin{align*}
	D(f_\theta)-D(f_{\theta_0}) = -\frac{1}{2} (\theta-\theta_0)^\top V(\theta-\theta_0) + \mathfrak{o}(\pnorm{\theta-\theta_0}{}^2).
	\end{align*}
	\item[(M2)] $\mathcal{F}=\{f_\theta: \theta \in \Theta\}$ admits a $P^m$-square integrable envelope $F$ such that
	\begin{align*}
	\int_0^{1} \big(\sup_Q \log\mathcal{N}\big(\epsilon \pnorm{F}{L_2(Q)},\mathcal{F}, L_2(Q)\big)\big)^{m/2}\ \d{\epsilon}<\infty.
	\end{align*}
	\item[(M3)] There exists a measurable map $\Delta: \mathcal{X}\to \R^d$ such that $P\Delta(X)=0$ and $P\pnorm{\Delta}{}^2<\infty$, and such that $\{r_n(\cdot,\theta):\theta \in \Theta\}$, defined by
	\begin{align*}
	r_n(x,\theta)\equiv \frac{\pi_1(f_\theta-f_{\theta_0})(x)-(\theta-\theta_0)\cdot \Delta(x)}{\pnorm{\theta-\theta_0}{}\vee n^{-1/2}},
	\end{align*}
	satisfy the following: for any $\delta_n\to 0$,
	\begin{align*}
	\sup_{\theta: \pnorm{\theta-\theta_0}{}\leq \delta_n} \bigabs{\G_n r_n(\cdot,\theta)}=\mathfrak{o}_{\mathbf{P}}(1).
	\end{align*}
\end{enumerate}
If $\pnorm{\hat{\theta}_n-\theta_0}{}=\mathfrak{o}_{\mathbf{P}}(1)$ and $\pnorm{\theta^\ast_n-\theta_0}{}=\mathfrak{o}_{P_\xi}(1)$ in $P_X$-probability, then $\sqrt{n}(\hat{\theta}_n-\theta_0)\rightsquigarrow_d m\cdot \mathcal{N}(0, V^{-1}\mathrm{cov}(\Delta) (V^{-1})^\top)$ and
\begin{align*}
\sup_{t \in \R^d} \bigabs{\Prob_{W|X}\big(\sqrt{n}(\theta_n^\ast-\hat{\theta}_n)\leq t\big)-\Prob\big(c\cdot\sqrt{n}(\hat{\theta}_n-\theta_0)\leq t\big) }\to_{P_X} 0.
\end{align*}
Here $c$ is the constant in (W2).
\end{theorem}

Condition (M1) requires that the population loss $D(f_\theta)$ is maximized at $\theta=\theta_0$ and an associated local Taylor expansion is valid with Hessian matrix $V$. Condition (M2) is a very typical requirement on the complexity of the model. Condition (M3) is a stochastic differentiability condition, where $\Delta$ is regarded as the derivative of $\pi_1(f_\theta)$ at $\theta=\theta_0$. 

Our conditions (M1)-(M3) are almost the same as the machinery in \cite[Theorem 5.5.7]{de2012decoupling} (see also \cite{arcones1994estimators}). Note that although our condition (M2) is stronger than \cite[condition (ii) of Theorem 5.5.7]{de2012decoupling}, there are few methods of checking (ii) in that theorem other than our (M2), so the examples studied therein can be applied quite immediately. 

One particularly interesting example is the simplicial median (cf. \cite{liu1990notion}) defined as follows: For any $(x_1,x_2,x_3) \in (\R^2)^3$, let $S(x_1,x_2,x_3)$ be the open triangle determined by $x_1,x_2,x_3$. For any $\theta \in \R^2$, let $f_\theta(x_1,x_2,x_3)\equiv \bm{1}_{C_\theta}(x_1,x_2,x_3)$ where $C_\theta\equiv \{(x_1,x_2,x_3) \in (\R^2)^3: \theta \in S(x_1,x_2,x_3)\}$. The simplicial median is defined as any maximizer of the map $\theta \mapsto U_n^{(3)}(f_\theta)$ over $\theta \in \Theta$, i.e. $\hat{\theta}_n \in \arg\max_{\theta \in \Theta} U_n^{(3)}(f_\theta)$. A central limit theorem for $\hat{\theta}_n$ is obtained in \cite{arcones1994estimators}, where the covariance of the normal limiting law depends on the law $P$ of the i.i.d. samples $X_1,X_2,\ldots$; see also \cite[Section 5.5.2]{de2012decoupling}. To apply Theorem \ref{thm:bootstrap_M}, the only `additional work' is to verify the slightly stronger condition (M2). This immediate follows as $\{\bm{1}_{C_\theta}: \theta \in \R^2\}$  is known to be a VC-subgraph class, see \cite[Example 5.2.4]{de2012decoupling}.

Our results here concerning bootstrap $M$-estimators can also be viewed as extensions of bootstrap theory for $M$-(or $Z$-) estimators under (i) the usual empirical measure studied in  \cite{wellner1996bootstrapping,bose2001generalised,chatterjee2005generalized,cheng2010bootstrap} and (ii) criteria functions that are convex with respect to the underlying parameter space, cf. \cite{bose2003generalized}.

\section{$M$-estimation under complex sampling}\label{section:complex_sampling_M}
In this section, we will study $M$-estimation under complex sampling designs. The exposition below largely follows \cite{han2018complex}. Let $U_N\equiv \{1,\ldots,N\}$, and $\mathcal{S}_N\equiv \{\{s_1,\ldots,s_n\}: n \leq N, s_i \in U_N, s_i\neq s_j,\forall i\neq j\}$ be the collection of subsets of $U_N$. We adopt the super-population framework as in \cite{rubin2005two}: Let $\{(X_i,Z_i) \in \mathcal{X}\times \mathcal{Z}\}_{i=1}^N$ be i.i.d. super-population samples defined on a probability space $(\mathcal{Y},\mathcal{A},\Prob_{(X,Z)})$, where $X^{(N)}\equiv (X_1,\ldots,X_N)$ is the vector of interest, and $Z^{(N)}\equiv (Z_1,\ldots,Z_N)$ is an auxiliary vector. A sampling design is a function $\mathfrak{p}: \mathcal{S}_N\times \mathcal{Z}^{\otimes N}\to [0,1]$ such that 
\begin{enumerate}
	\item for all $s \in \mathcal{S}_N$, $z^{(N)}\mapsto \mathfrak{p}(s,z^{(N)})$ is measurable,
	\item for all $z^{(N)} \in \mathcal{Z}^{\otimes N}$, $s\mapsto \mathfrak{p}(s, z^{(N)})$ is a probability measure.
\end{enumerate}

The probability space we work with that includes both the super-population and the design-space is the same product space $({\mathcal{S}}_N\times \mathcal{Y},  \sigma({\mathcal{S}}_N)\times \mathcal{A}, \Prob)$ as constructed in \cite{boistard2017functional}. We include the construction here for convenience of the reader: the probability measure $\Prob$ is uniquely defined through its restriction on all rectangles: for any $s\times E \in {\mathcal{S}}_N\times \mathcal{A}$,
\begin{align*}
\Prob\left(s\times E\right)\equiv \int_E \mathfrak{p}(s, z^{(N)}(\omega))\ \d{\Prob_{(X,Z)}(\omega)}\equiv \int_E \Prob_d(s,\omega)\ \d{\Prob_{(X,Z)}(\omega)}.
\end{align*}
We also use $P$ to denote the marginal law for $X$ for notational convenience.

Given $(X^{(N)},Z^{(N)})$ and a sampling design $\mathfrak{p}$, let $\{\xi_i\}_{i=1}^N\subset [0,1]$ be random variables defined on $({\mathcal{S}}_N\times \mathcal{Y},  \sigma({\mathcal{S}}_N)\times \mathcal{A}, \Prob)$ with  $\pi_i\equiv \pi_i(Z^{(N)})\equiv \E[\xi_i|Z^{(N)}]$. We further assume that $\{\xi_i\}_{i=1}^N$ are independent of $X^{(N)}$ conditionally on $Z^{(N)}$. Typically we take $\xi_i\equiv \bm{1}_{i \in s}$, where $s\sim \mathfrak{p}$, to be the indicator of whether or not the $i$-th sample $X_i$ is observed (and in this case $\pi_i(Z^{(N)})=\sum_{s \in \mathcal{S}_N: i \in s} \mathfrak{p}(s,Z^{(N)})$), but we do not require this structure a priori. $\pi_i$'s are often referred to be the first-order inclusion probabilities, and $\pi_{ij}\equiv \pi_{ij}(Z^{(N)})\equiv \E[\xi_i\xi_j|Z^{(N)}](i\neq j)$ are the second-order inclusion probabilities.

\begin{assumption}\label{assumption:sampling_design}
Consider the following conditions on the sampling design $\mathfrak{p}$:

\, (B1) $\min_{1\leq i\leq N}\pi_i\geq \pi_0>0$.

\, (B2-LLN) $\frac{1}{N}\sum_{i=1}^N \big(\frac{\xi_i}{\pi_i}-1\big) =\mathfrak{o}_{\mathbf{P}}(1)$.
\end{assumption}

(B1) is a common assumption in the literature. (B2-LLN) says that the weights $\{\xi_i/\pi_i\}$ satisfy a law of large numbers. For various sampling designs satisfying Assumption \ref{assumption:sampling_design}, including sampling without replacement, Bernoulli sampling, rejective/high entropy sampling, stratified sampling (with and without overlaps), etc., we refer the reader to \cite{han2018complex}.

Under the complex sampling setting, it is natural to use the following (inverse-weighted) $M$-estimator based on univariate kernels
\begin{align*}
\hat{\theta}_N^{\pi}\in   \arg\max_{\theta \in \Theta} \sum_{i=1}^N \frac{\xi_i}{\pi_i} f_\theta(X_i)
\end{align*}
that maximizes the Horvitz-Thompson weighted empirical measure over $\{f_\theta:\theta \in \Theta\}$. For multivariate kernels, it is natural to consider the following generalization:
\begin{align}
\hat{\theta}_N^{\pi}\in  \arg\max_{\theta \in \Theta} \sum_{i_1\neq \ldots\neq i_m} \frac{\xi_{i_1}}{\pi_{i_1}}\cdots \frac{\xi_{i_m}}{\pi_{i_m}} f_\theta(X_{i_1},\ldots,X_{i_m}).
\end{align}
We let $\Prob_N^\pi(f)\equiv \frac{1}{N}\sum_{i=1}^N \frac{\xi_i}{\pi_i} f(X_i)$ and $\G_N^\pi(f)\equiv \sqrt{N}(\Prob_N^\pi-P)(f)$ denote the Hortivz-Thompson empirical measure and empirical process respectively.

Our first main result in this section is the following.

\begin{theorem}\label{thm:M_estimation_sampling_CLT}
Suppose Assumption \ref{assumption:sampling_design}, and conditions (M1)-(M3) in Theorem \ref{thm:bootstrap_M} hold. Then
\begin{align*}
\sqrt{N}\big(\hat{\theta}_N^{\pi}-\theta_0\big) = mV^{-1} \G_N^\pi \Delta+\mathfrak{o}_{\mathbf{P}}(1).
\end{align*}
\end{theorem}

For a general sampling design, the asymptotic distribution of $\G_N^\pi \Delta$ is not entirely a trivial problem. We refer the reader to \cite[Proposition 1]{han2018complex} for a summary for the asymptotic variance (more generally, the covariance structure of the limit of $\G_N^\pi$) for various sampling designs illustrated above.

In Theorem \ref{thm:M_estimation_sampling_CLT} we considered a finite-dimensional $M$-estimation problem. It is also possible to consider $M$-estimation problem in an infinite-dimensional setting based on Horvitz-Thompson weighted $U$-statistics:
\begin{align}\label{def:ERM}
\hat{f}_N^\pi \equiv \arg\min_{f \in \mathcal{F}} \sum_{i_1\neq \ldots\neq i_m} \frac{\xi_{i_1}}{\pi_{i_1}}\cdots \frac{\xi_{i_m}}{\pi_{i_m}} f(X_{i_1},\ldots,X_{i_m}),
\end{align}
where $\mathcal{F}$ is a class of symmetric non-degenerate (and typically non-negative) kernels. The quality of the estimator defined in (\ref{def:ERM}) is evaluated through the \emph{excess risk} of $\hat{f}_N^\pi$, denoted $\mathcal{E}_P(\hat{f}_N^\pi)$, where 
\begin{align*}
\mathcal{E}_P(f)\equiv Pf-\inf_{g \in \mathcal{F}} Pg,\quad \forall f \in \mathcal{F}.
\end{align*}
The problem of studying excess risk of empirical risk minimizers under the usual empirical measure has been extensively studied in the 2000s; we only refer the reader to \cite{gine2006concentration,koltchinskii2006local} and references therein. The paper \cite{clemencon2008ranking} extended the scope of ERM to criteria functions based on $U$-statistics of order 2 under the i.i.d. sampling. Our goal here will be a study of the excess risk for the $M$-estimator based on Horvitz-Thompson weighted $U$-statistics as defined in (\ref{def:ERM}) for the general empirical risk minimization problem under general sampling designs.

To this end, let $\mathcal{F}_{\mathcal{E}}(\delta)\equiv \{f \in \mathcal{F}: \mathcal{E}_P(f)<\delta^2\}$, let $\rho_P:\mathcal{F}\times \mathcal{F}\to \R_{\geq 0}$ be such that $\rho_P^2(f,g)\geq P(f-g)^2-\big(P(f-g)\big)^2$, and $D(\delta)\equiv \sup_{f,g \in \mathcal{F}_{\mathcal{E}}(\delta)} \rho_P(f,g)$. 

Now we may state our second main result of this section.
\begin{theorem}\label{thm:M_estimation_sampling_excess_risk}
Suppose Assumption \ref{assumption:sampling_design} holds. Suppose that there exists some $L>0,\kappa\geq 1$ such that 
\begin{align}\label{cond:low_noise}
D(\delta)\leq L\delta^{1/\kappa}.
\end{align}
Further assume that $\mathcal{F}$ is a uniformly bounded VC-subgraph class. Then for any $t,s,u\geq 0$, if
\begin{align*}
r_N\geq K_1 \bigg[\bigg(\frac{\log N}{N}\bigg)^{\frac{\kappa}{4\kappa-2}}+ \bigg(\frac{s\vee t^2}{N}\bigg)^{\frac{\kappa}{4\kappa-2} }+\bigg(\frac{s\vee u}{N}\bigg)^{1/2}\bigg],
\end{align*}
we have
\begin{align*}
& \Prob\big(\mathcal{E}_P(\hat{f}_N^\pi)\geq r_N^2\big)\\
&\leq K_2\big(e^{-s/K_2}/s+ e^{-u^{2/m}/K_2}\big)+\Prob\bigg(\biggabs{\frac{1}{\sqrt{N}}\sum_{i=1}^N\bigg(\frac{\xi_i}{\pi_i}-1\bigg) }>t\bigg).
\end{align*}
Here the constants $K_1,K_2>0$ only depend on $m,\pi_0,\kappa$.
\end{theorem}

Condition (\ref{cond:low_noise}) is comparable to \cite[Assumption 4]{clemenccon2016learning} in the case $m=2$. This condition is well-understood for the usual empirical risk minimization problems, typically under the name of `low-noise' condition, cf. \cite{mammen1999smooth,tsybakov2004optimal}. In particular, if $\kappa$ is close to $1$, then a faster rate than the standard $\sqrt{N}$ rate is possible.

Specializing our result to the case $m=2$ and i.i.d. sampling, we recover \cite[Corollary 6]{clemencon2008ranking}. It is easy to see from the proofs that $\mathcal{F}$ being a VC-subgraph class is not a crucial assumption. Indeed one can replace it with more general super-polynomial uniform entropy conditions with slight modifications of the proofs. We omit these digressions here.

\section{Proofs for Section \ref{section:multiplier_ineq}}\label{section:proof_1}

\subsection{Proof of Theorem \ref{thm:multiplier_ineq}}
\begin{proof}[Proof of Theorem \ref{thm:multiplier_ineq}]
	Since the class $\mathcal{F}$ contains degenerate kernels of order $m-1$, conditional on $\bm{\xi}$, by decoupling (cf. \cite[Theorem 3.1.1]{de2012decoupling}) and symmetrization, we have with $C_m \equiv 2^m \prod_{k=2}^m (k^k-1)$ (as in \cite[Theorem 3.1.1]{de2012decoupling})
	\begin{align}\label{ineq:multiplier_ineq_2}
	&\E \biggpnorm{\sum_{1\leq i_1,\ldots,i_m\leq n} \xi_{i_1}\cdots\xi_{i_m} f(X_{i_1},\ldots,X_{i_m})}{\mathcal{F}}\nonumber\\
    &\leq C_m\cdot \E \biggpnorm{\sum_{1\leq i_1,\ldots,i_m\leq n} \xi_{i_1}\cdots\xi_{i_m} f(X_{i_1}^{(1)},\ldots,X_{i_m}^{(m)})}{\mathcal{F}}\nonumber\\
	&\leq 2^{m} C_m\cdot  \E \biggpnorm{\sum_{1\leq i_1,\ldots,i_m\leq n} \xi_{i_1}\cdots\xi_{i_m}\epsilon_{i_1}^{(1)}\cdots\epsilon_{i_m}^{(m)} f(X_{i_1}^{(1)},\ldots,X_{i_m}^{(m)})}{\mathcal{F}}\nonumber\\
	& = 2^{m} C_m\cdot \E \bigg\lVert\sum_{1\leq i_1,\ldots,i_m\leq n} \abs{\xi_{i_1} }\cdots\abs{\xi_{i_m}} \nonumber\\
	&\qquad\qquad\times \mathrm{sgn}(\xi_{i_1})\epsilon_{i_1}^{(1)}\cdots\mathrm{sgn}(\xi_{i_m})\epsilon_{i_{m}}^{(m)} f(X_{i_1}^{(1)},\ldots,X_{i_m}^{(m)})\bigg\rVert_{\mathcal{F}}.
	\end{align}
	Note here in the second inequality where the symmetrization is carried out according to the degeneracy level of $\mathcal{F}$ due to \cite[Remark 3.5.4 (2)]{de2012decoupling}. The constant $2^{m}$ appears by tracking the constant in the arguments in \cite[pp. 140]{de2012decoupling}. Since $(\mathrm{sgn}(\xi_{1})\epsilon_{1}^\cdot,\ldots, \mathrm{sgn}(\xi_{n})\epsilon_{n}^\cdot)$ is independent of $(\xi_{1},\ldots,\xi_n)$ and has the same distribution as $(\epsilon_1^\cdot,\ldots,\epsilon_n^\cdot)$, we have
	\begin{align}\label{ineq:multiplier_ineq_3}
	&\E \biggpnorm{\sum_{1\leq i_1,\ldots,i_m\leq n} \xi_{i_1}\cdots\xi_{i_m} f(X_{i_1},\ldots,X_{i_m})}{\mathcal{F}} \nonumber\\
	& \leq 2^{m} C_m\cdot  \E \biggpnorm{\sum_{1\leq i_1,\ldots,i_m\leq n} \abs{\xi_{i_1}}\cdots\abs{\xi_{i_m}} \epsilon_{i_1}^{(1)}\cdots \epsilon_{i_m}^{(m)} f(X_{i_1}^{(1)},\ldots,X_{i_m}^{(m)})}{\mathcal{F}}.
	\end{align}
	Let $\abs{\xi_{(1)}}\geq \ldots\geq \abs{\xi_{(n)}}$ be the reversed order statistics of $\{\abs{\xi_i}\}_{i=1}^n$, and $\pi$ be a permutation over $\{1,\ldots,n\}$ such that $\abs{\xi_i}=\abs{\xi_{(\pi(i))} }$. By the invariance of $(P_\epsilon \otimes P)^{mn}$ and the fact that $\bm{\xi}$ is independent of $\bm{X}^\cdot,\bm{\epsilon}^\cdot$, we have that
	\begin{align}
	&\E_{\bm{\epsilon},\bm{X}} \biggpnorm{\sum_{1\leq i_1,\ldots,i_m\leq n} \abs{\xi_{i_1}}\cdots\abs{\xi_{i_m}} \epsilon_{i_1}^{(1)}\cdots \epsilon_{i_m}^{(m)} f(X_{i_1}^{(1)},\ldots,X_{i_m}^{(m)})}{\mathcal{F}}\nonumber\\
	& =\E_{\bm{\epsilon},\bm{X}} \biggpnorm{\sum_{1\leq i_1,\ldots,i_m\leq n} \abs{\xi_{(\pi(i_1))}}\cdots\abs{\xi_{(\pi(i_m)) }} \epsilon_{i_1}^{(1)}\cdots \epsilon_{i_m}^{(m)} f(X_{i_1}^{(1)},\ldots,X_{i_m}^{(m)})}{\mathcal{F}}\nonumber\\
	& = \E_{\bm{\epsilon},\bm{X}} \biggpnorm{\sum_{1\leq i_1,\ldots,i_m\leq n} \abs{\xi_{(i_1)}}\cdots\abs{\xi_{(i_m) }} \epsilon_{\pi^{-1}(i_1)}^{(1)}\cdots \epsilon_{\pi^{-1}(i_m)}^{(m)} f(X_{\pi^{-1}(i_1)}^{(1)},\ldots,X_{\pi^{-1}(i_m)}^{(m)})}{\mathcal{F}}\nonumber\\
	& = \E_{\bm{\epsilon},\bm{X}} \biggpnorm{\sum_{1\leq i_1,\ldots,i_m\leq n} \abs{\xi_{(i_1)}}\cdots\abs{\xi_{(i_m) }} \epsilon_{i_1}^{(1)}\cdots \epsilon_{i_m}^{(m)} f(X_{i_1}^{(1)},\ldots,X_{i_m}^{(m)})}{\mathcal{F}}. 
	\end{align}
	Using $\abs{\xi_{(i)}} = \sum_{\ell\geq i} (\abs{\xi_{(\ell)} }-\abs{\xi_{(\ell+1)} })$ (with $\abs{\xi_{(n+1)}}\equiv 0$) and combining (\ref{ineq:multiplier_ineq_2})-(\ref{ineq:multiplier_ineq_3}), we have that
	\begin{align*}
	&\big(2^{m} C_m\big)^{-1}\cdot \E \biggpnorm{\sum_{1\leq i_1,\ldots,i_m\leq n} \xi_{i_1}\cdots\xi_{i_m} f(X_{i_1},\ldots,X_{i_m})}{\mathcal{F}}\\
	&\leq \E \bigg\lVert\sum_{1\leq i_1,\ldots,i_m\leq n} \sum_{\ell_k\geq i_k, 1\leq k\leq m} (\abs{\xi_{(\ell_1)}}-\abs{\xi_{(\ell_1+1)} })\cdots(\abs{\xi_{(\ell_m)}}-\abs{\xi_{(\ell_m+1)}}) \nonumber\\
	&\qquad\qquad\qquad\qquad\qquad\qquad \times \epsilon_{i_1}^{(1)}\cdots \epsilon_{i_m}^{(m)} f(X_{i_1}^{(1)},\ldots,X_{i_m}^{(m)}) \bigg\lVert_{\mathcal{F}}\nonumber\\
	&\leq  \E \bigg[ \sum_{1\leq \ell_1,\ldots,\ell_m\leq n}(\abs{\xi_{(\ell_1)}}-\abs{\xi_{(\ell_1+1)}})\cdots(\abs{\xi_{(\ell_m)}}-\abs{\xi_{(\ell_m+1)}}) \nonumber\\
	&\qquad\qquad\qquad\qquad \times \E  \biggpnorm{\sum_{\substack{1\leq i_k\leq \ell_k, 1\leq k\leq m} } \epsilon_{i_1}^{(1)}\cdots\epsilon_{i_m}^{(m)} f(X_{i_1}^{(1)},\ldots,X_{i_m}^{(m)})}{\mathcal{F}} \bigg]\nonumber\\
	&\leq  \E \bigg[ \sum_{1\leq \ell_1,\ldots,\ell_m\leq n}  \int^{\abs{\xi_{(\ell_1)}}}_{\abs{\xi_{(\ell_1+1)} }}\cdots \int^{\abs{\xi_{(\ell_m)}}}_{\abs{\xi_{(\ell_m+1)} })}\psi_n(\ell_1,\ldots,\ell_m)\ \d{t_m}\cdots\d{t_1}\bigg]\nonumber\\
	&\leq \E \bigg[ \sum_{1\leq \ell_1,\ldots,\ell_m\leq n}  \int^{\abs{\xi_{(\ell_1)}}}_{\abs{\xi_{(\ell_1+1)}}}\cdots \int^{\abs{\xi_{(\ell_m)}}}_{\abs{\xi_{(\ell_m+1)}})}\\
	&\qquad\qquad\qquad\psi_n(\abs{\{i: \abs{\xi_i}> t_1\}},\ldots,\abs{\{i: \abs{\xi_i}> t_m\}})\ \d{t_m}\cdots\d{t_1}\bigg]\nonumber\\
	&\leq \E \bigg[\int_{\R_{\geq 0}^m }    \psi_n(\abs{\{i: \abs{\xi_i}> t_1\}},\ldots,\abs{\{i: \abs{\xi_i}> t_m\}})\ \d{t_1}\cdots\d{t_m}\bigg].\nonumber
	\end{align*}
	In the second inequality in the above display we changed the order of the summation. The first claim now follows from Fubini's theorem.

      Now suppose that $\psi_n(\ell_1,\ldots,\ell_m)= \bar{\psi}_n(\prod_{k=1}^m \ell_k)$. Then we may further bound the above display by
      \begin{align*}
      & \int_{\R_{\geq 0}^m }   \E \bar{\psi}_n \bigg(\prod_{k=1}^m \sum_{i=1}^n \bm{1}_{\abs{\xi_i}>t_k}\bigg)\ \d{t_1}\ldots\d{t_m}\\
      & = \int_{\R_{\geq 0}^m }   \E \bar{\psi}_n \bigg(  \sum_{1\leq i_1,\ldots,i_m\leq n} \prod_{k=1}^m \bm{1}_{\abs{\xi_{i_k}}>t_k}\bigg)\ \d{t_1}\ldots\d{t_m}\\
      &\leq \int_{\R_{\geq 0}^m }   \bar{\psi}_n \bigg(  \sum_{1\leq i_1,\ldots,i_m\leq n} \E \prod_{k=1}^m \bm{1}_{\abs{\xi_{i_k}}>t_k}\bigg)\ \d{t_1}\ldots\d{t_m} \quad \textrm{(by Jensen's inequality)}\\
      &\leq \int_{\R_{\geq 0}^m }   \bar{\psi}_n \bigg(  \sum_{1\leq i_1,\ldots,i_m\leq n}  \prod_{k=1}^m \Prob\big(\abs{\xi_{i_k}}>t_k\big)^{1/m}\bigg)\ \d{t_1}\ldots\d{t_m},
      \end{align*}
      where the last inequality follows from generalized H\"older's inequality and the assumption that $\bar{\psi}_n$ is non-decreasing.
\end{proof}

\subsection{Proof of Proposition \ref{prop:lower_bound_multiplier_ineq}}

\begin{proof}[Proof of Proposition \ref{prop:lower_bound_multiplier_ineq}]
Let $\alpha = 2/(\gamma-1)$. By \cite[Lemma 6]{han2017sharp}, $\bar{\mathcal{F}}_1\equiv C^{1/\alpha}([0,1])$ is an $\alpha$-full class on $[0,1]$. Let $\mathcal{F}_1$ be the $P$-centered version of $\bar{\mathcal{F}}_{1}$, i.e. $\mathcal{F}_{1}\equiv \{f-Pf: f \in \bar{\mathcal{F}}_{1}\}$. Take $\mathcal{F}\equiv  \{f(x_1,\ldots,x_m)=\prod_{k=1}^{m}\varphi_k(x_k): \varphi_k \in \mathcal{F}_{1}, P\varphi_k^2\leq n^{-2/(2+\alpha)}, 1\leq k\leq m\}$. By Lemma \ref{lem:gine_koltchinskii_matching_bound_ep}, for any $1\leq k\leq m$, and $1\leq \ell_1,\ldots,\ell_m\leq n$, we have
\begin{align}\label{ineq:lower_bound_multiplier_ineq_1}
\E \sup_{ \substack{\varphi_k \in \mathcal{F}_{1}:\\ P\varphi_k^2\leq \ell_k^{-2/(2+\alpha)}}} \biggabs{\sum_{i_k=1}^{\ell_k} \epsilon_{i_k}^k \varphi_k(X_{i_k}^k)}\asymp \sqrt{\ell_k} \big(\ell_k^{-1/(2+\alpha)}\big)^{1-\alpha/2} = \ell_k^{\frac{\alpha}{2+\alpha}}.
\end{align}
Hence for all $1\leq \ell_1,\ldots,\ell_m\leq n$,
\begin{align*}
&\E \biggpnorm{\sum_{\substack{1\leq i_k\leq \ell_k, 1\leq k\leq m}} \epsilon_{i_1}^{(1)}\cdots\epsilon_{i_m}^{(m)} f(X_{i_1}^{(1)},\ldots,X_{i_m}^{(m)})}{\mathcal{F}}\\
& = \prod_{k=1}^m \E \sup_{ \substack{\varphi_k \in \mathcal{F}_{1}:\\ P\varphi_k^2\leq n^{-2/(2+\alpha)}}} \biggabs{\sum_{i_k=1}^{\ell_k} \epsilon_{i_k}^k \varphi_k(X_{i_k}^k)}\\
&\leq \prod_{k=1}^m \E \sup_{ \substack{\varphi_k \in \mathcal{F}_{1}:\\ P\varphi_k^2\leq \ell_k^{-2/(2+\alpha)}}} \biggabs{\sum_{i_k=1}^{\ell_k} \epsilon_{i_k}^k \varphi_k(X_{i_k}^k)}\leq C_{\alpha}\prod_{k=1}^m \ell_k^{\frac{\alpha}{2+\alpha}}= C_{\alpha}  \prod_{k=1}^m \ell_k^{1/\gamma}.
\end{align*}
This proves the upper bound. Next we consider the lower bound. Let $\{\xi'_i\}$ be an independent copy of $\{\xi_i\}$. Then
\begin{align*}
&\E \biggpnorm{\sum_{1\leq i_1,\ldots,i_m\leq n} \xi_{i_1}\cdots\xi_{i_m} f(X_{i_1},\ldots,X_{i_m})}{\mathcal{F}}  \\
&= \prod_{k=1}^m\E \bigg[ \sup_{ \substack{\varphi_k \in \mathcal{F}_{1}:\\ P\varphi_k^2\leq n^{-2/(2+\alpha)}}} \biggabs{\sum_{i=1}^{n} \xi_{i} \varphi_k(X_{i})}\bigg]\\
& \geq  \prod_{k=1}^m \frac{1}{2}\cdot \E \bigg[ \sup_{ \substack{\varphi_k \in \mathcal{F}_{1}:\\ P\varphi_k^2\leq n^{-2/(2+\alpha)}}} \biggabs{\sum_{i=1}^{n} \big(\xi_{i}-\xi'_i\big) \varphi_k(X_{i})}\bigg] \quad ( \textrm{by triangle inequality}) \\
& = \frac{1}{2^m} \prod_{k=1}^m\E \bigg[ \sup_{ \substack{\varphi_k \in \mathcal{F}_{1}:\\ P\varphi_k^2\leq n^{-2/(2+\alpha)}}} \biggabs{\sum_{i=1}^{n} \epsilon_i \abs{\xi_{i}-\xi'_i} \varphi_k(X_{i})}\bigg] \quad ( \textrm{by symmetry of }\xi_i-\xi_i')\\
& \geq \frac{1}{2^m} \prod_{k=1}^m\E \bigg[ \sup_{ \substack{\varphi_k \in \mathcal{F}_{1}:\\ P\varphi_k^2\leq n^{-2/(2+\alpha)}}} \biggabs{\sum_{i=1}^{n} \epsilon_i \E\abs{\xi_{i}-\xi'_i} \varphi_k(X_{i})}\bigg] \quad ( \textrm{by Jensen})\\
&\geq \frac{\pnorm{\xi_1}{1}}{2^m} \prod_{k=1}^m\E \bigg[ \sup_{ \substack{\varphi_k \in \mathcal{F}_{1}:\\ P\varphi_k^2\leq n^{-2/(2+\alpha)}}} \biggabs{\sum_{i=1}^{n} \epsilon_i  \varphi_k(X_{i})}\bigg] \gtrsim n^{m/\gamma},
\end{align*}
where in the last line we used (\ref{ineq:lower_bound_multiplier_ineq_1}) with $\ell_1,\ldots,\ell_m=n$, and the fact that $\E\abs{\xi_i-\xi_i'}\geq \E \abs{\xi_i-\E \xi_i'}= \E \abs{\xi_i}=\pnorm{\xi_1}{1}$ for all $i=1,\ldots,n$. 
\end{proof}

\subsection{Proof of Corollary \ref{cor:asymptotic_equicont}}

\begin{proof}[Proof of Corollary \ref{cor:asymptotic_equicont}]
	Take $\psi_n(\ell_1,\ldots,\ell_m)\equiv a(\ell_1,\ldots,\ell_m) \big(\prod_{k=1}^m \ell_k\big)^{1/2}$. By Theorem \ref{thm:multiplier_ineq},
	\begin{align*}
	&\E \biggpnorm{\sum_{1\leq i_1,\ldots,i_m\leq n} \xi_{i_1}\cdots\xi_{i_m} f(X_{i_1},\ldots,X_{i_m})}{\mathcal{F}_{(n,\ldots,n),n}}\\
	&\leq K_m \int_{\R_{\geq 0}^m}\E \bigg[a\bigg(\sum_{i=1}^n \bm{1}_{\abs{\xi_i}>t_1},\ldots,\sum_{i=1}^n \bm{1}_{\abs{\xi_i}>t_m}\bigg) \prod_{k=1}^m \bigg(\sum_{i=1}^n \bm{1}_{\abs{\xi_i}>t_k}\bigg)^{1/2}\bigg] \ \d{t_1}\cdots\d{t_m}\\
	&\leq K_m \int_{\R_{\geq 0}^m} A_{2,n}(t_1,\ldots,t_m) \bigg\{\E \prod_{k=1}^m \sum_{i=1}^n \bm{1}_{\abs{\xi_i}>t_k}\bigg\}^{1/2} \ \d{t_1}\cdots\d{t_m}\\
	& \leq K_m \int_{\R_{\geq 0}^m} A_{2,n}(t_1,\ldots,t_m) \bigg(\sum_{1\leq i_1,\ldots,i_m\leq n}\prod_{k=1}^m \Prob\big(\abs{\xi_{i_k}}>t_k\big)^{1/m}\bigg)^{1/2} \ \d{t_1}\cdots\d{t_m}\\
	& = n^{m/2} K_m \int_{\R_{\geq 0}^m} A_{2,n}(t_1,\ldots,t_m)  \prod_{k=1}^m \Prob\big(\abs{\xi_{1}}>t_k\big)^{1/2m} \ \d{t_1}\cdots\d{t_m}.
	\end{align*}
	Here
	\begin{align*}
	A_{2,n}(t_1,\ldots,t_m)\equiv  \bigg\{\E \bigg[a^2\bigg(\sum_{i=1}^n \bm{1}_{\abs{\xi_i}>t_1},\ldots,\sum_{i=1}^n \bm{1}_{\abs{\xi_i}>t_m}\bigg)\bigg] \bigg\}^{1/2} \to 0
	\end{align*}
	as long as none of $\{\Prob\big(\abs{\xi_{1}}>t_k\big): 1\leq k\leq m\}$ vanishes. The claim now follows from dominated convergence theorem.
\end{proof}

\section{Proofs for Section \ref{section:multiplier_CLT_bootstrap}}\label{section:proof_2}

\subsection{Proof of Theorem \ref{thm:multiplier_CLT}}

\begin{proof}[Proof of Theorem \ref{thm:multiplier_CLT}]
	We only need to check the asymptotic equi-continuity. For any $\delta>0$, let $\mathcal{F}_{\delta}\equiv \{f-g: f,g \in \mathcal{F}, \pnorm{f-g}{L_2(P^m)}<\delta\}$. For any $f \in \mathcal{F}$, let $\tilde{f}(x_1,\ldots,x_m)\equiv f(x_1,\ldots,x_m)$ if $x_1\neq \cdots\neq x_m$ and $0$ otherwise, and let $\tilde{\mathcal{F}}\equiv \{\tilde{f}: f \in \mathcal{F}\}$. Then the $L_2$ distance associated to the conditional (partially decoupled) Rademacher chaos process
	\begin{align*}
	\bigg\{\frac{1}{(\prod_{k=1}^m \ell_k)^{1/2}}\sum_{1\leq i_k\leq \ell_k, 1\leq k\leq m} \epsilon_{i_1}^{(1)}\cdots\epsilon_{i_m}^{(m)} f(X_{i_1},\ldots,X_{i_m}): f \in \mathcal{F}\bigg\lvert X_1,\ldots,X_n\bigg\}
	\end{align*}
	is given by 
	\begin{align*}
	e_{\bm{\ell}}^2(f,g)\equiv \frac{1}{\prod_{k=1}^m \ell_k}\sum_{1\leq i_k\leq \ell_k, 1\leq k\leq m} (f-g)^2(X_{i_1},\ldots,X_{i_m}).
	\end{align*}
	Let $\pnorm{f}{\bm{\ell}}^2\equiv e_{\bm{\ell}}^2(f,0)$, and $r_{\bm{\ell}}(\delta)\equiv \sup_{f \in \tilde{\mathcal{F}}_\delta} \pnorm{f}{\bm{\ell}}^2$. By the entropy maximal inequality for Rademacher chaos process (cf.  \cite[Corollary 5.1.8]{de2012decoupling}), we have
	\begin{align}\label{ineq:multiplier_clt_1}
	&\E_{\bm{\epsilon}} \biggpnorm{\frac{1}{(\prod_{k=1}^m \ell_k)^{1/2}}\sum_{1\leq i_k\leq \ell_k, 1\leq k\leq m} \epsilon_{i_1}^{(1)}\cdots\epsilon_{i_m}^{(m)} f(X_{i_1},\ldots,X_{i_m}) }{\tilde{\mathcal{F}}_{\delta}}\nonumber\\
	&\leq C_1 \int_0^{r_{\bm{\ell}}(\delta) } \big(\log \mathcal{N}\big(\epsilon, \mathcal{F}, e_{\bm{\ell}}\big)\big)^{m/2}\ \d{\epsilon}\nonumber\\
	& = C_1 \pnorm{F}{\bm{\ell}}\cdot \int_0^{ r_{\bm{\ell}}(\delta)/\pnorm{F}{\bm{\ell}} } \big(\log\mathcal{N}(\epsilon \pnorm{F}{\bm{\ell}},\mathcal{F}, e_{\bm{\ell}})\big)^{m/2}\ \d{\epsilon}\nonumber\\
	& \leq C_1 \pnorm{F}{\bm{\ell}}\cdot \int_0^{ r_{\bm{\ell}}(\delta)/\pnorm{F}{\bm{\ell}} } \big(\sup_Q \log\mathcal{N}\big(\epsilon \pnorm{F}{L_2(Q)},\mathcal{F}, L_2(Q)\big)\big)^{m/2}\ \d{\epsilon}.
	\end{align}
	Without loss of generality we may take $F\geq 1$ so the upper bound in the integral can be replaced by $r_{\bm{\ell}}(\delta)$. By Proposition \ref{prop:LLN_unbalanced_Ustat}, $\pnorm{F}{\bm{\ell}} \to_p \pnorm{F}{L_2(P)}$ as $\ell_1\wedge\ldots\wedge \ell_m \to \infty$, and hence by the integrability on the far right hand side of (\ref{ineq:multiplier_clt_1}) it suffices to show that $r_{\bm{\ell}}(\delta) \to_p 0$ as $\ell_1\wedge\ldots\wedge \ell_m \to \infty$ followed by $\delta \to 0$. Clearly it only remains to show that
	\begin{align}\label{ineq:multiplier_clt_2}
	\sup_{f \in \tilde{\mathcal{F}}_\delta} \biggabs{\frac{1}{\prod_{k=1}^m \ell_k}\sum_{1\leq i_k\leq \ell_k, 1\leq k\leq m} \big( f^2(X_{i_1},\ldots,X_{i_m})-P^m f^2\big)}\to_p 0
	\end{align}
	as $\ell_1\wedge\ldots\wedge \ell_m \to \infty$. To this end we verify (\ref{cond:LLN_unbalanced_Ustat}) in Proposition \ref{prop:LLN_unbalanced_Ustat}. We only do this for $k=m$. Note that $e_{\bm{\ell},j'}$ (introduced in the statement of Proposition \ref{prop:LLN_unbalanced_Ustat}) can be bounded by the $L_1$ distance corresponding to the uniform measure on the (random set) $\{(X_{i_1},\ldots,X_{i_m}): 1\leq i_j\leq \ell_j, 1\leq j\leq m\}$, and hence by the $L_2$ distance $e_{\bm{\ell}}$ (cf. Remark \ref{rmk:relate_e_j_to_l2}). Furthermore it is easy to verify that $\mathcal{N}(\delta, \mathcal{F}^2_M, L_2(Q))\leq \mathcal{N}(\delta/2M, \mathcal{F}_M, L_2(Q))$. Hence
	\begin{align*}
	&\max_{1\leq j'\leq m} \E  \bigg( \frac{\log \mathcal{N}(\delta, \mathcal{F}^2_M, e_{\bm{\ell},j'})}{\ell_{j'}} \bigg)^{1/2}\\
	&\leq (\delta/2M)^{-1} (\ell_1\wedge \cdots \wedge \ell_m)^{-1/2} \E \bigg[\int_0^{\delta/2M} \big(\log \mathcal{N}(\epsilon, \mathcal{F}_M, e_{\bm{\ell}})\big)^{m/2}\ \d{\epsilon}\bigg]\\
	& \leq (\delta/2M)^{-1} (\ell_1\wedge \cdots \wedge \ell_m)^{-1/2}\\
	&\qquad\qquad\int_0^{1} \big(\sup_Q\log \mathcal{N}(\epsilon\pnorm{F}{L_2(Q)}, \mathcal{F}, L_2(Q))\big)^{m/2}\ \d{\epsilon}\cdot \pnorm{F}{L_2(P^m)}\to 0
	\end{align*}
	as long as $\ell_1\wedge \cdots \wedge \ell_m \to \infty$. Hence (\ref{cond:LLN_unbalanced_Ustat}) is verified and Proposition \ref{prop:LLN_unbalanced_Ustat} applies to conclude that (\ref{ineq:multiplier_clt_2}) holds. Combined with (\ref{ineq:multiplier_clt_1}) and decoupling inequality (cf.  \cite[Theorem 3.5.3]{de2012decoupling}), we have shown that for any $\{\delta_{\bm{\ell}}\}$ such that $\delta_{\bm{\ell}}\to 0$ as $\ell_1\wedge \cdots \wedge \ell_m \to \infty$, there exists some sequence $\{a_{\bm{\ell}}\}$ with $a_{\bm{\ell}}\to 0$ as $\ell_1\wedge \cdots \wedge \ell_m \to \infty$ such that
	\begin{align*}
	\E \biggpnorm{\sum_{1\leq i_k\leq \ell_k, 1\leq k\leq m} \epsilon_{i_1}^{(1)}\cdots\epsilon_{i_m}^{(m)} f(X_{i_1}^{(1)},\ldots,X_{i_m}^{(m)})}{\tilde{\mathcal{F}}_{\delta_{\bm{\ell}}}}\leq a_{\bm{\ell}}\bigg(\prod_{k=1}^m \ell_k\bigg)^{1/2}.
	\end{align*}
    Now for any $\{\delta_n\}$ such that $\delta_n \searrow 0$, let $\delta_{\bm{\ell}}\equiv \delta_{\max_k \ell_k}$. Then for any $1\leq \ell_1,\ldots,\ell_m\leq n$, $\tilde{\mathcal{F}}_{\delta_{\bm{\ell}}} = \tilde{\mathcal{F}}_{\delta_{\max_k \ell_k}}\supset \tilde{\mathcal{F}}_{\delta_n}$. The above display holds for such constructed $\{\delta_{\bm{\ell}}\}$. Apply Corollary \ref{cor:asymptotic_equicont} we obtain
	\begin{align*}
	&\E \biggpnorm{n^{-m/2}\sum_{1\leq i_1,\ldots,i_m\leq n} \xi_{i_1}\cdots\xi_{i_m} f(X_{i_1},\ldots,X_{i_m})}{\tilde{\mathcal{F}}_{\delta_n}} \to 0.
	\end{align*}
	This completes the proof for the asymptotic equi-continuity.
\end{proof}

\subsection{Proof of Theorem \ref{thm:bootstrap_CLT}} 

\begin{proof}[Proof of Theorem \ref{thm:bootstrap_CLT}]
	We first prove finite-dimensional convergence. By Cram\'er-Wold and countability of $\mathcal{F}$, we only need to show that for any $f \in L_2^{c,m}(P)$,
	\begin{align}\label{ineq:bootstrap_U_process_1}
	\sup_{\psi \in \mathrm{BL}}\biggabs{\E\bigg[\psi\bigg( \binom{n}{m}^{1/2}\tilde{U}_{n,\xi}^{(m)}(f)\bigg)\bigg\lvert \{X_i\}\bigg]-\E \psi(c\cdot K_P(f)) }\to 0\textrm{ a.s.}
	\end{align}
	By \cite[(4.2.5), page 175]{de2012decoupling} and \cite[Section 2A]{arcones1992bootstrap}, any $f \in L_2^{c,m}(P)$ can be expanded in $L_2(P^m)$ by $f=\sum_{q=1}^\infty c_q h_m^{\psi_q}$, where $\{c_q\}$ is a sequence of real numbers, and $h_m^{\psi_q}(x_1,\ldots,x_m)\equiv \psi_q(x_1)\cdots\psi_q(x_m)$ for some bounded $\psi_q \in L_2^{c,1}(P)$. Fix $\epsilon>0$. Then there exists $Q_\epsilon \in \N$ such that with $f_\epsilon\equiv \sum_{q=1}^{Q_\epsilon} c_q h_m^{\psi_q}$, $\pnorm{f-f_\epsilon}{L_2(P^m)}\leq \epsilon$. The left hand side of (\ref{ineq:bootstrap_U_process_1}) can be further bounded by
	\begin{align}\label{ineq:bootstrap_U_process_2}
	&\sup_{\psi \in \mathrm{BL}}\biggabs{\E\bigg[\psi\bigg( \binom{n}{m}^{1/2}\tilde{U}_{n,\xi}^{(m)}(f)\bigg)\bigg\lvert \{X_i\}\bigg]-\E \psi(c\cdot K_P(f)) }\nonumber\\
	&\leq \sup_{\psi \in \mathrm{BL}}\biggabs{ \E\bigg[\psi\bigg( \binom{n}{m}^{1/2}\tilde{U}_{n,\xi}^{(m)}(f)\bigg)\bigg\lvert \{X_i\}\bigg]-\E\bigg[\psi\bigg( \binom{n}{m}^{1/2}\tilde{U}_{n,\xi}^{(m)}(f^\epsilon)\bigg)\bigg\lvert \{X_i\}\bigg] }\nonumber\\
	&\qquad +\sup_{\psi \in \mathrm{BL}}\biggabs{\E\bigg[\psi\bigg( \binom{n}{m}^{1/2}\tilde{U}_{n,\xi}^{(m)}(f^\epsilon)\bigg)\bigg\lvert \{X_i\}\bigg]-\E \psi(c\cdot K_P(f^\epsilon)) } \nonumber\\
	&\qquad\qquad +\sup_{\psi \in \mathrm{BL}}\bigabs{\E \psi(c\cdot K_P(f^\epsilon))-\E \psi(c\cdot K_P(f)) }\nonumber\\
	& \equiv (I)+(II)+(III).
	\end{align}
	For notational convenience, we let $\bar{f}_\epsilon\equiv f-f_\epsilon$ and $\E^{X}[\cdot]\equiv \E[\cdot|\{X_i\}]$. For the first term in (\ref{ineq:bootstrap_U_process_2}), using the Lipschitz property of $\psi$ and the fact that $\psi$ is bounded by $1$, we have
	\begin{align*}
	(I)^2&\leq \E^X \biggabs{2\wedge\binom{n}{m}^{1/2}\tilde{U}_{n,\xi}^{(m)}(\bar{f}^\epsilon) }^2\\
	&\lesssim  \E^X \bigg(1\wedge n^{-m/2}\sum_{1\leq i_1<\ldots<i_m\leq n} (\xi_{i_1}-1)\cdots(\xi_{i_m}-1) \bar{f}_\epsilon(X_{i_1},\ldots,X_{i_m})\bigg)^2\\
	&\lesssim \E^X_{\xi} \E_R \bigg(1\wedge n^{-m/2}\sum_{1\leq i_1\neq \ldots\neq i_m\leq n} (\xi_{R_{i_1}}-1)\cdots(\xi_{R_{i_m}}-1) \bar{f}_\epsilon(X_{i_1},\ldots,X_{i_m})\bigg)^2\\
	&\lesssim \sum_{ \substack{\alpha_i \in \{1,2\}: \sum_{i=1}^l \alpha_i =2m,\\ \alpha_1\geq\ldots\geq \alpha_l, 1\leq l\leq m} } \E_\xi^X \bigg[1\wedge n^{-m}\E_R \bigg[\prod_{i=1}^l \big(\xi_{R_i}-1\big)^{\alpha_i}\bigg] \\
	&\qquad\qquad\qquad\times \sum_{\substack{i_1\neq \ldots\neq i_m,\\ i_1'\neq \ldots\neq i_m', \\ i_j=i_j', 1\leq j\leq \max\{j:\alpha_j=2\} }} \bar{f}_\epsilon(X_{i_1},\ldots,X_{i_m})\bar{f}_\epsilon(X_{i_1'},\ldots,X_{i_m'}) \bigg]\\
	&\lesssim \sum_{ \substack{\alpha_i \in \{1,2\}: \sum_{i=1}^l \alpha_i =2m,\\ \alpha_1\geq\ldots\geq \alpha_l, 1\leq l\leq m} }   \E \bigg[1\wedge \frac{1}{n}\sum_{i=1}^n \big(\xi_i-1\big)^2\bigg]^{m} \\
	&\qquad\qquad\qquad\times n^{-l}\sum_{\substack{i_1\neq \ldots\neq i_m,\\ i_1'\neq \ldots\neq i_m', \\ i_j=i_j', 1\leq j\leq \max\{j:\alpha_j=2\} }} \bar{f}_\epsilon(X_{i_1},\ldots,X_{i_m})\bar{f}_\epsilon(X_{i_1'},\ldots,X_{i_m'}) ,
	\end{align*}
	where the last inequality follows from Lemma \ref{lem:moment_perm_prod_weight}. By the usual law of large number for $U$-statistics (cf. \cite[Theorem 4.1.4]{de2012decoupling}), we have
	\begin{align*}
	&n^{-l}\sum_{\substack{i_1\neq \ldots\neq i_m,\\ i_1'\neq \ldots\neq i_m', \\ i_j=i_j', 1\leq j\leq \max\{j:\alpha_j=2\} }} \bar{f}_\epsilon(X_{i_1},\ldots,X_{i_m})\bar{f}_\epsilon(X_{i_1'},\ldots,X_{i_m'})\\
	&\to_{\textrm{a.s.}} \E \bar{f}_\epsilon(X_{1},\ldots,X_{m}) \bar{f}_\epsilon(X_{1}',\ldots,X_{m}')\\
	&\qquad\qquad\qquad\qquad (\textrm{where }X_j=X_j'\textrm{ for } 1\leq j\leq\max\{j:\alpha_j=2\})\\
	&\leq  P^m\bar{f}_\epsilon^2\leq \epsilon^2.
	\end{align*}
	Combining the above two displays, we obtain
	\begin{align}\label{ineq:bootstrap_U_process_3}
	\limsup_{n \to \infty} (I) \lesssim_{m,\xi} \epsilon, \textrm{ a.s.}
	\end{align}
	Next we handle the second term in (\ref{ineq:bootstrap_U_process_2}). Note that
	\begin{align*}
	&\binom{n}{m}^{1/2}\tilde{U}_{n,\xi}^{(m)}(f^\epsilon)\\
	&=_d \frac{1}{\binom{n}{m}^{1/2}}\sum_{q=1}^{Q_\epsilon} c_q \sum_{1\leq i_1<\ldots<i_m\leq n} (\xi_{R_{i_1}}-1)\cdots(\xi_{R_{i_m}}-1) \psi_q(X_{i_1})\cdots\psi_q(X_{i_m})\\
	& = \frac{n^{m/2}}{\binom{n}{m}^{1/2}}\sum_{q=1}^{Q_\epsilon} c_qR_m\bigg(\frac{1}{n^{1/2}}\sum_{i=1}^n (\xi_{R_i}-1)\psi_q(X_i), \ldots,\frac{1}{n^{m/2}}\sum_{i=1}^n (\xi_{R_i}-1)^m\psi_q^m(X_i) \bigg)\\
	&\equiv (1+\mathfrak{o}(1))(m!)^{1/2}\sum_{q=1}^{Q_\epsilon} c_q R_m(A_{n,q}^{(1)},\ldots,A_{n,q}^{(m)}),
	\end{align*}
	where $R_m$ is determined through (\ref{eqn:Newton_identity}). Below we determine the limits of $A_{n,q}^{(\ell)}$, $\ell=1,2,3,\ldots,m$.
	\begin{itemize}
		\item[$(\ell=1)$] Apply Lemma \ref{lem:exchangeable_CLT} with $a_i\equiv \psi_q(X_i)-\Prob_n\psi_q$ and $\xi_i$ replaced by $\xi_{R_i}-1$ in our setting, we see that $A_{n,q}^{(1)}\rightsquigarrow_d c\cdot G_P(\psi_q)$ a.s. 
		\item[$(\ell=2)$] Note that
		\begin{align*}
		\E_R^{X,\xi} (A_{n,q}^{(2)})=\frac{1}{n}\sum_{i=1}^n(\xi_i-1)^2\cdot \frac{1}{n}\sum_{i=1}^n \psi_q^2(X_i)\to_{P_\xi} c^2 \E\psi_q^2, \textrm{ a.s.}
		\end{align*}
		Furthermore, 
		\begin{align*}
		&\mathrm{Var}_R^{X,\xi}(A_{n,q}^{(2)})\\
		&= \E_{R}^{{X},{\xi}} \big(A_{n,q}^{(2)}\big)^2 - \left(\E_{R}^{{X},{\xi}} \big(A_{n,q}^{(2)}\big)\right)^2 \\
		& = \E_{R}^{{X},{\xi}} \left[ \frac{1}{n}\sum_{i=1}^n \big(\xi_i-1\big)^2 \psi_q^2(X_{R_i})\right]^2 - \left[ \frac{1}{n}\sum_{i=1}^n \big(\xi_i-1\big)^2 \Prob_n \psi_q^2\right]^2\\
		& = \frac{1}{n^2} \sum_{i,j}\big(\xi_i-1\big)^2\big(\xi_j-1\big)^2\left[\E_{R}^{{X}} \psi_q^2(X_{R_i}) \psi_q^2(X_{R_j}) - (\Prob_n \psi_q^2)^2\right]\\
		& = \frac{1}{n^2}\sum_i \big(\xi_i-1\big)^4 \left[\E_{R}^{{X}} \psi_q^4(X_{R_i})-(\Prob_n \psi_q^2)^2\right]\\
		&\quad + \frac{1}{n^2} \sum_{i\neq j} \big(\xi_i-1\big)^2\big(\xi_j-1\big)^2 \left[\E_{R}^{{X}} \psi_q^2(X_{R_i}) \psi_q^2(X_{R_j}) - (\Prob_n \psi_q^2)^2\right]\\ 
		& \leq  \frac{1}{n^2} \sum_i \big(\xi_i-1\big)^4\cdot \Prob_n \psi_q^4 +\frac{1}{n^2}\bigg(\sum_i (\xi_i-1)^2\bigg)^2\cdot \frac{1}{n-1} \Prob_n \psi_q^4\\
		& \lesssim \frac{1}{n^2} \sum_{i=1}^n \big(\xi_i-1\big)^4 \cdot \Prob_n \psi_q^4\\
		&\lesssim \pnorm{\psi_q}{\infty}^4 \frac{\max_i (\xi_i-1)^2}{n}\cdot \frac{1}{n}\sum_{i=1}^n\big(\xi_i-1)^2 \to_{P_\xi} 0, \textrm{ a.s.}
		\end{align*}
		The first inequality in the above display follows since
		\begin{align*}
		&\E_{R}^{{X}} \psi_q^2(X_{R_i}) \psi_q^2(X_{R_j}) - (\Prob_n \psi_q^2)^2\\
		& = \frac{1}{n(n-1)}\left[\sum_{i\neq j} \psi_q^2(X_{i}) \psi_q^2(X_{j})\right] -(\Prob_n \psi_q^2)^2\\
		& \leq \frac{1}{n-1} (\Prob_n \psi_q^2)^2\leq \frac{1}{n-1} \Prob_n \psi_q^4.
		\end{align*}
		This shows that $A_{n,q}^{(2)}\to_{P_\xi} c^2 \E\psi_q^2$ a.s.
		\item[$(\ell\geq 3)$] Note that
		\begin{align*}
		\E_R^{X,\xi} \abs{A_{n,q}^{(\ell)}}&\leq \frac{1}{n^{\ell/2}}\sum_{i=1}^n\abs{\xi_i-1}^\ell\cdot \frac{1}{n}\sum_{i=1}^n \abs{\psi_q(X_i)}^\ell\\
		&\leq  \bigg(\frac{\max_i \abs{\xi_i-1}^2}{n}\bigg)^{\frac{\ell-2}{2}} \cdot \frac{1}{n}\sum_{i=1}^n \abs{\xi_i-1}^2\cdot \pnorm{\psi_q}{\infty}\\
		&\to_{P_\xi} 0,\textrm{ a.s.}
		\end{align*}
		This shows that $A_{n,q}^{(\ell)}\to_{P_\xi} 0$ a.s.
	\end{itemize}
	We have thus shown $R_m(A_{n,q}^{(1)},\ldots,A_{n,q}^{(m)})\rightsquigarrow_d R_m(G_P(c\psi_q), \E(c\psi_q)^2,0,\ldots,0)= c(m!)^{-1/2}\cdot K_P(\psi_q)$ a.s. By linearity of $K_p$, it follows that $\binom{n}{m}^{1/2}\tilde{U}_{n,\xi}^{(m)}(f^\epsilon)\rightsquigarrow_d c\cdot K_P(f^\epsilon)$ a.s. Hence
	\begin{align}\label{ineq:bootstrap_U_process_4}
	\lim_{n \to \infty} (II)=0, \textrm{ a.s.}
	\end{align}
	For the third term in (\ref{ineq:bootstrap_U_process_2}), note that
	\begin{align}\label{ineq:bootstrap_U_process_5}
	(III) &\leq c \sqrt{\E K_P^2(\bar{f}^\epsilon)}\to 0 \quad (\epsilon \to 0)
	\end{align}
	by the definition of $K_P$ (cf. \cite[page 176]{de2012decoupling}). Combining (\ref{ineq:bootstrap_U_process_2})-(\ref{ineq:bootstrap_U_process_5}) and taking the limits as $n\to \infty$ followed by $\epsilon\to 0$, we see that (\ref{ineq:bootstrap_U_process_1}) holds, and hence proving the finite-dimensional convergence. 
	
	For asymptotic equi-continuity, we need to prove that for any $\epsilon>0$ and $\delta_n \to 0$, $
	\Prob_\xi^X \big(\bigpnorm{\tilde{U}_{n,\xi}^{(m)}(f)}{\mathcal{F}_{\delta_n}}>\epsilon\big)\to_{P_X^\ast} 0$. 
	Hence it suffices to prove that $
	\Prob \big(\bigpnorm{\tilde{U}_{n,\xi}^{(m)}(f)}{\mathcal{F}_{\delta_n}}>\epsilon\big)\to 0$, 
	or even the stronger $
	\E \bigpnorm{\tilde{U}_{n,\xi}^{(m)}(f)}{\mathcal{F}_{\delta_n}}\to 0$. 
	This can be checked using similar arguments as the proofs of Theorem \ref{thm:multiplier_CLT}, and hence completing the proof.
\end{proof}

\section{Proofs for Section \ref{section:bootstrap_M}}\label{section:proof_3}

\begin{proof}[Proof of Theorem \ref{thm:bootstrap_M}]
	For notational convenience, let
	\begin{align*}
	D_{n,\xi}(f)&\equiv D_{n,\xi}^{(m)}(f)\equiv \frac{1}{m!\binom{n}{m}}\sum_{i_1\neq \ldots \neq i_m}\xi_{i_1}\cdots \xi_{i_m}  f(X_{i_1},\ldots,X_{i_m}),\\
	D_{n}(f)&\equiv D_{n}^{(m)}(f)\equiv \frac{1}{m!\binom{n}{m}}\sum_{i_1\neq \ldots \neq i_m} f(X_{i_1},\ldots,X_{i_m}),
	\end{align*}
	and $D(f) \equiv P^m f$. We claim the following:
	\begin{enumerate}
		\item[(Claim 1)] $\sqrt{n}\pnorm{\theta_n^\ast-\theta_0}{}=\mathcal{O}_{P_\xi}(1)$ in $P_X$-probability, and $\sqrt{n}\pnorm{\hat{\theta}_n-\theta_0}{}=\mathcal{O}_{\mathbf{P}}(1)$.
		\item[(Claim 2)] For any $\delta_n\to 0$, 
		\begin{align*}
		& \max_{2\leq k\leq m}\E\sup_{\theta: \pnorm{\theta-\theta_0}{}\leq \delta_n} \biggabs{n^{-k+1}\sum_{i_1\neq \ldots\neq i_k} \xi_{i_1}\cdots\xi_{i_k} \pi_k(f_{\theta}-f_{\theta_0})(X_{i_1},\ldots,X_{i_k}) }\to 0.
		\end{align*}
	\end{enumerate}
	Proofs of these claims will be deferred towards the end of the proof. Then with $g_\theta\equiv f_{\theta}-f_{\theta_0}$ and $\bar{g}_{\theta} = f_{\theta}-f_{\theta_0}- P^m(f_{\theta}-f_{\theta_0})$, we have
	\begin{align}\label{ineq:bootstrap_M_0}
	&n\big(D_{n,\xi}(f_{\theta^\ast_n})-D_{n,\xi}(f_{{\theta}_0})\big)\nonumber\\
	& = n\big(D(f_{\theta^\ast_n})-D(f_{{\theta}_0})\big) + n D_{n,\xi} \big[(f_{\theta_n^\ast}-f_{\theta_0})- P^m (f_{\theta_n^\ast}-f_{\theta_0})\big] \nonumber\\
	&\qquad\qquad\qquad\qquad\qquad  + n  \big(D_{n,\xi}- D\big)  \big(P^m (f_{\theta_n^\ast}-f_{\theta_0})\big) \nonumber\\
	& = -\frac{1}{2}n(\theta_n^\ast-\theta_0)^\top V (\theta_n^\ast-\theta_0) \nonumber\\
	&\qquad\qquad +  \frac{n}{m!\binom{n}{m}}\sum_{i_1\neq \ldots \neq i_m}\xi_{i_1}\cdots \xi_{i_m} \big\{(f_{\theta^\ast_n}-f_{{\theta}_0})(X_{i_1},\ldots,X_{i_m})-P^m(f_{\theta_n^\ast}-f_{\theta_0})\big\}\nonumber\\
	&\qquad\qquad + \frac{n}{m! \binom{n}{m}} \sum_{i_1\neq \ldots\neq i_m} (\xi_{i_1}\cdots \xi_{i_m}-1) P^m(f_{\theta_n^\ast}-f_{\theta_0})+\mathfrak{o}_{\mathbf{P}}(1)\nonumber\\
	&\qquad\qquad\qquad\qquad\qquad\qquad\qquad\qquad\qquad \textrm{(by assumption (M1) and Claim 1)}\nonumber\\
	& = -\frac{1}{2}n\big({\theta}_n^{\ast}-\theta_0\big)^\top V \big({\theta}_n^{\ast}-\theta_0\big) \nonumber\\
	&\qquad\qquad + \frac{n}{m!\binom{n}{m}} \sum_{i_1\neq \ldots \neq i_m} \xi_{i_1}\cdots\xi_{i_m} \sum_{j=1}^m \pi_1\big(\bar{g}_{{\theta}_n^{\ast}}\big)(X_{i_j})\nonumber\\
	&\qquad\qquad + \frac{n}{m! \binom{n}{m}} \sum_{i_1\neq \ldots\neq i_m} \xi_{i_1}\cdots\xi_{i_m} \sum_{2\leq k\leq m}\sum_{j_1<\ldots<j_k} \pi_k\big(\bar{g}_{{\theta}_n^{\pi}}\big)(X_{i_{j_1}},\ldots,X_{i_{j_k}})\nonumber\\
	&\qquad\qquad +\bigg(\frac{1}{m!\binom{n}{m}}\sum_{i_1\neq \ldots\neq i_m}\xi_{i_1}\ldots\xi_{i_m}-1\bigg)\cdot n\big(D(f_{{\theta}_n^{\pi}})-D(f_{\theta_0})\big) + \mathfrak{o}_{\mathbf{P}}(1)\nonumber\\
	& = -\frac{1}{2}n\big({\theta}_n^{\ast}-\theta_0\big)^\top V \big({\theta}_n^{\pi}-\theta_0\big) \nonumber \\
	&\qquad\qquad + \frac{n}{(m-1)!\binom{n}{m}} \sum_{i_1=1}^n \xi_{i_1} \pi_1\big(\bar{g}_{{\theta}_n^{\ast}}\big)(X_{i_1})\bigg(\sum_{i_2,\ldots,i_m: i_1\neq i_2\neq \ldots \neq i_m} \xi_{i_2}\cdots \xi_{i_m}\bigg)\nonumber\\
	&\qquad\qquad + \sum_{2\leq k\leq m} C_{m,k} \frac{n}{\binom{n}{m}} \sum_{i_1\neq\ldots\neq i_k} \xi_{i_1}\cdots \xi_{i_k} \pi_k\big(\bar{g}_{{\theta}_n^{\ast}}\big)(X_{i_1},\ldots,X_{i_k})\nonumber\\
	&\qquad\qquad\qquad\qquad\qquad\qquad\times\bigg(\sum_{i_{k+1},\ldots,i_m: i_1\neq i_2\neq \ldots \neq i_m} \xi_{i_{k+1}}\cdots \xi_{i_m}\bigg)\nonumber\\
	&\qquad\qquad + \bigg(\frac{1}{m!\binom{n}{m}}\sum_{i_1\neq \ldots\neq i_m}\xi_{i_1}\ldots\xi_{i_m}-1\bigg)\cdot n\big(D(f_{{\theta}_n^{\ast}})-D(f_{\theta_0})\big)+\mathfrak{o}_{\mathbf{P}}(1)\nonumber\\
	& = -\frac{1}{2}n\big({\theta}_n^{\ast}-\theta_0\big)^\top V \big({\theta}_n^{\ast}-\theta_0\big)  + (I)+(II)+(III)+\mathfrak{o}_{\mathbf{P}}(1).
	\end{align}

	For $(I)$ in (\ref{ineq:bootstrap_M_0}), note that
	\begin{align}\label{ineq:bootstrap_M_01}
	(I)&=\frac{n}{(m-1)!\binom{n}{m}} \sum_{i_1=1}^n \xi_{i_1} \pi_1\big(\bar{g}_{{\theta}_n^{\ast}}\big)(X_{i_1})\nonumber\\
	&\qquad\qquad\times \bigg(\sum_{i_2,\ldots,i_m: i_2\neq \ldots \neq i_m} \xi_{i_2}\cdots \xi_{i_m}-\sum_{j=2}^m \sum_{\substack{i_2,\ldots,i_m:\\ i_j = i_1, i_2\neq \ldots\neq i_m} }\xi_{i_2}\cdots\xi_{i_j}\cdots\xi_{i_m}\bigg)\nonumber\\
	& = \sum_{i_1=1}^n \xi_{i_1} \pi_1\big(\bar{g}_{{\theta}_n^{\ast}}\big)(X_{i_1})\nonumber\\
	&\qquad\qquad \times \bigg(\frac{n}{(m-1)!\binom{n}{m}} \sum_{i_2,\ldots,i_m: i_2\neq \ldots \neq i_m} \xi_{i_2}\cdots \xi_{i_m} + \mathcal{O}_{\mathbf{P}}(n^{-1}) \bigg) \nonumber\\
	& =  \sum_{i_1=1}^n \xi_{i_1} \pi_1\big(\bar{g}_{{\theta}_n^{\ast}}\big)(X_{i_1})\nonumber\\
	&\qquad\qquad\times \bigg(\frac{n^m}{\binom{n}{m}}\cdot R_{m-1}\bigg(\frac{1}{n}\sum_{i=1}^n\xi_i,\ldots,\frac{1}{n^{m-1}}\sum_{i=1}^n \xi_i^{m-1}\bigg)+\mathcal{O}_{\mathbf{P}}(n^{-1})\bigg)\nonumber\\
	& = \big(1+\mathfrak{o}_{\mathbf{P}}(1)\big)m\sum_{i_1=1}^n \xi_{i_1} \pi_1\big(\bar{g}_{{\theta}_n^{\ast}}\big)(X_{i_1}).
	\end{align}
	Here in the last equality we used Assumption \ref{assumption:weights} and the fact that $R_{m-1}(1,0,\ldots,0)=1/(m-1)!$. For $(II)$ in (\ref{ineq:bootstrap_M_0}), note that
	\begin{align}\label{ineq:bootstrap_M_02}
	(II)& = \sum_{2\leq k\leq m} C_{m,k} \frac{n}{\binom{n}{m}} \sum_{i_1\neq\ldots\neq i_k} \xi_{i_1}\cdots \xi_{i_k} \pi_k\big(\bar{g}_{{\theta}_n^{\ast}}\big)(X_{i_1},\ldots,X_{i_k})\\
	&\qquad\qquad\times\bigg(\sum_{i_{k+1},\ldots,i_m: i_1\neq i_2\neq \ldots \neq i_m} \xi_{i_{k+1}}\cdots \xi_{i_m}\bigg)\nonumber\\
	& = \sum_{2\leq k\leq m} C_{m,k}' \bigg(n^{-(k-1)}\sum_{i_1\neq\ldots\neq i_k} \xi_{i_1}\cdots \xi_{i_k} \pi_k\big(\bar{g}_{{\theta}_n^{\ast}}\big)(X_{i_1},\ldots,X_{i_k})\bigg)\nonumber\\
	&\qquad\qquad \times \bigg(n^{-(m-k)}\sum_{i_{k+1},\ldots,i_m: i_{k+1}\neq \ldots \neq i_m} \xi_{i_{k+1}}\cdots \xi_{i_m}+\mathfrak{o}_{\mathbf{P}}(1)\bigg)\nonumber\\
	& = \mathfrak{o}_{\mathbf{P}}(1).\nonumber
	\end{align}
	Here in the last line we used Claim 2.
	
	For $(III)$ in (\ref{ineq:bootstrap_M_0}), note that
	\begin{align}\label{ineq:bootstrap_M_03}
	(III)& =\bigg(\frac{n^m}{\binom{n}{m}}\cdot R_m\bigg(\frac{1}{n}\sum_{i=1}^n\xi_i, \frac{1}{n^2}\sum_{i=1}^n \xi_i^2,\ldots,\frac{1}{n^m}\sum_{i=1}^n \xi_i^m\bigg)-1\bigg)\cdot \mathcal{O}_{\mathbf{P}}(1) \nonumber\\
	& = \mathfrak{o}_{\mathbf{P}}(1). 
	\end{align}
	Combining (\ref{ineq:bootstrap_M_0})-(\ref{ineq:bootstrap_M_02}), we see that
	\begin{align*}
	&n\big(D_{n,\xi}(f_{\theta^\ast_n})-D_{n,\xi}(f_{{\theta}_0})\big)\\
	& = -\frac{1}{2}n\big({\theta}_n^{\ast}-\theta_0\big)^\top V \big({\theta}_n^{\ast}-\theta_0\big) +\big(1+\mathfrak{o}_{\mathbf{P}}(1)\big) m\sum_{i=1}^n \xi_{i} \pi_1\big(\bar{g}_{{\theta}_n^{\ast}}\big)(X_{i})+\mathfrak{o}_{\mathbf{P}}(1).
	\end{align*}	
	Since $\pi_1(\bar{g}_{\theta})=\pi_1(f_\theta-f_{\theta_0})(x) = (\theta-\theta_0)\cdot \Delta(x)+(\pnorm{\theta-\theta_0}{}\vee n^{-1/2}) r_n(x,\theta)$, it follows that
	\begin{align*}
	&\sum_{i=1}^n \xi_i \pi_1(f_{\theta^\ast_n}-f_{{\theta}_0})(X_{i}) \\
	& = (\theta^\ast_n-\theta_0) \cdot \sum_{i=1}^n\xi_i \Delta(X_i) + \big[\pnorm{\theta^\ast_n-\theta_0}{}\vee n^{-1/2}\big] \sum_{i=1}^n\xi_i r_n(x,\theta^\ast_n)\\
	& = (\theta^\ast_n-\theta_0) \cdot\sum_{i=1}^n\xi_i \Delta(X_i) + \mathfrak{o}_{\mathbf{P}}(1).
	\end{align*}
	Here we used the assumption (M3), $\pnorm{\theta^\ast_n-\theta_0}{}=\mathcal{O}_{\mathbf{P}}(n^{-1/2})$ and the multiplier inequality Theorem \ref{thm:multiplier_ineq} with $m=1$ (see also \cite[Theorem 1 ]{han2017sharp}) to conclude that $\bigabs{\big[\pnorm{\theta^\ast_n-\theta_0}{}\vee n^{-1/2}\big] \sum_{i=1}^n\xi_i r_n(x,\theta^\ast_n)}=\mathfrak{o}_{\mathbf{P}}(1)$.
	
	Combining the above displays, we have
	\begin{align}\label{ineq:bootstrap_M_1}
	&n\big(D_{n,\xi}(f_{\theta^\ast_n})-D_{n,\xi}(f_{{\theta}_0})\big)\nonumber\\
	&= -\frac{1}{2}n(\theta_n^\ast-\theta_0)^\top V (\theta_n^\ast-\theta_0)+ n(\theta^\ast_n-\theta_0)\cdot  \Delta_n^\ast+\mathfrak{o}_{\mathbf{P}}(1),
	\end{align}
	where $\Delta_n^\ast = m\cdot \frac{1}{n}\sum_{i=1}^n \xi_i \Delta(X_i)$. Expand $n\big(D_{n,\xi}(f_{\theta_0+V^{-1}\Delta_n^\ast })-D_{n,\xi}(f_{{\theta}_0})\big)$, we have
	\begin{align}\label{ineq:bootstrap_M_2}
	&n\big(D_{n,\xi}(f_{\theta_0+V^{-1}\Delta_n^\ast })-D_{n,\xi}(f_{{\theta}_0})\big) \nonumber\\
	&= -\frac{1}{2}n(\Delta_n^\ast)^\top V^{-1} \Delta_n^\ast + n(V^{-1}\Delta_n^\ast)\cdot \Delta_n^\ast + \mathfrak{o}_{\mathbf{P}}(1)\nonumber\\
	& = \frac{1}{2}n(\Delta_n^\ast)^\top V^{-1} \Delta_n^\ast + \mathfrak{o}_{\mathbf{P}}(1).
	\end{align}
	Combining (\ref{ineq:bootstrap_M_1})-(\ref{ineq:bootstrap_M_2}), we have
	\begin{align*}
	&\frac{1}{2}n(\Delta_n^\ast)^\top V^{-1} \Delta_n^\ast + \mathfrak{o}_{\mathbf{P}}(1)\\
	 &= n\big(D_{n,\xi}(f_{\theta_0+V^{-1}\Delta_n^\ast })-D_{n,\xi}(f_{{\theta}_0})\big)\\
	&\leq n\big(D_{n,\xi}(f_{\theta^\ast_n})-D_{n,\xi}(f_{{\theta}_0})\big)\\
	&=-\frac{1}{2}n(\theta_n^\ast-\theta_0)^\top V (\theta_n^\ast-\theta_0)+ n(\theta^\ast_n-\theta_0)\cdot \Delta_n^\ast +\mathfrak{o}_{\mathbf{P}}(1),\nonumber
	\end{align*}
	which is equivalent to
	\begin{align*}
	n\bigpnorm{V^{1/2}(\theta_n^\ast-\theta_0)-V^{-1/2}\Delta_n^\ast}{}^2 = \mathfrak{o}_{P}(1)\Leftrightarrow \sqrt{n}(\theta_n^\ast-\theta_0)=V^{-1}\sqrt{n}\Delta_n^\ast+\mathfrak{o}_{\mathbf{P}}(1).
	\end{align*}
	On the other hand, expanding $n\big(D_n(f_{\hat{\theta}_n})-D_n(f_{\theta_0})\big)$ yields that $\sqrt{n}(\hat{\theta}_n-\theta_0) = V^{-1}\sqrt{n}\Delta_n+\mathfrak{o}_{\mathbf{P}}(1)$, where $\Delta_n\equiv m\cdot \frac{1}{n}\sum_{i=1}^n \Delta(X_i)$, and hence 
	\begin{align*}
	\sqrt{n}(\theta_n^\ast-\hat{\theta}_n) = V^{-1}\sqrt{n} \Delta_n^\xi + \mathfrak{o}_{\mathbf{P}}(1)\equiv V^{-1}\cdot m\cdot  \frac{1}{\sqrt{n}}\sum_{i=1}^n (\xi_i-1) \Delta(X_i) + \mathfrak{o}_{\mathbf{P}}(1).
	\end{align*}
	By Lemma \ref{lem:exchangeable_CLT}, $\sqrt{n} \Delta_n^\xi$ is asymptotically normal with covariance matrix $m^2c^2\cdot\mathrm{Cov}(\Delta)$ in $P_X$-probability, while $\sqrt{n}\Delta_n$ has asymptotic covariance $m^2\cdot \mathrm{Cov}(\Delta)$. The theorem then follows from Lemma \ref{lem:transfer_probability} and \cite[Lemma 2.11]{van2000asymptotic}, modulo the claims made in the beginning, the proofs of which we will present now.
	
	First we prove Claim 1. To this end, let $\lambda\equiv \lambda_{\mathrm{min}}(V)/4>0$. Then
	\begin{align*}
	&\lambda n \pnorm{\theta_n^\ast-\theta_0}{}^2\\
	&\leq \frac{1}{2} n (\theta_n^\ast-\theta_0)^\top V (\theta_n^\ast-\theta_0)-\lambda n \pnorm{\theta_n^\ast-\theta_0}{}^2(\equiv Z_n)\\
	&\leq Z_n\bm{1}_{\frac{1}{2}(\theta_n^\ast-\theta_0)^\top V (\theta_n^\ast-\theta_0)-D(f_{\theta_0})+D(f_{\theta_n^\ast})\leq \lambda \pnorm{\theta_n^\ast-\theta_0}{}^2 }\\
	&\qquad + Z_n\bm{1}_{\frac{1}{2}(\theta_n^\ast-\theta_0)^\top V (\theta_n^\ast-\theta_0)-D(f_{\theta_0})+D(f_{\theta_n^\ast})> \lambda \pnorm{\theta_n^\ast-\theta_0}{}^2 }\\
	&\leq -n\big(D(f_{\theta_n^\ast})-D(f_{\theta_0})\big)+\mathfrak{o}_{\mathbf{P}}(1)\qquad \textrm{(by assumption (M1))}\\
	&= n(D_{n,\xi}-D_n)(f_{\theta^\ast_n}-f_{{\theta}_0})+ n(D_{n}-D)(f_{\theta^\ast_n}-f_{{\theta}_0})\\
	&\qquad -n D_{n,\xi}(f_{\theta^\ast_n}-f_{{\theta}_0})+\mathfrak{o}_{\mathbf{P}}(1)\\
	&\leq n(D_{n,\xi}-D_n)(f_{\theta^\ast_n}-f_{{\theta}_0})+ n(D_{n}-D)(f_{\theta^\ast_n}-f_{{\theta}_0})+\mathfrak{o}_{\mathbf{P}}(1)\\
	&\qquad\qquad\qquad\qquad\qquad\qquad\qquad\qquad (\textrm{by definition of }\theta_n^\ast)\\
	&= n(\theta^\ast_n-\theta_0)\cdot \big(\Delta_n^\ast+\Delta_n\big)+\mathfrak{o}_{\mathbf{P}}(1)\\
	&\qquad\qquad\qquad\qquad (\textrm{by similar derivations as in } (\ref{ineq:bootstrap_M_0})-(\ref{ineq:bootstrap_M_1}) \textrm{ using Claim 2})\\
	&= \mathcal{O}_{\mathbf{P}}(1)\cdot \sqrt{n}\pnorm{\theta_n^\ast-\theta_0}{}+\mathfrak{o}_{\mathbf{P}}(1).
	\end{align*}
	Solving for a quadratic inequality we obtain $\sqrt{n}\pnorm{\theta_n^\ast-\theta_0}{}=\mathcal{O}_{\mathbf{P}}(1)$, and hence by Lemma \ref{lem:transfer_probability}, $\sqrt{n}\pnorm{\theta_n^\ast-\theta_0}{}=\mathcal{O}_{P_\xi}(1)$ in $P_X$-probability. Similar arguments conclude that $\sqrt{n}\pnorm{\hat{\theta}_n-\theta_0}{}=\mathcal{O}_{\mathbf{P}}(1)$ and hence Claim 1 is proved.
	
	Next we prove Claim 2. We only need to show that for any $\{\delta_{\bm{\ell}}\}$ such that $\delta_{\bm{\ell}}\to 0$ as $\ell_1\wedge\ldots\wedge \ell_k\to 0$, there exists some uniformly bounded sequence $a_{\bm{\ell}}$ with $a_{\bm{\ell}}\to 0$ as $\ell_1\wedge\ldots\wedge \ell_k\to 0$ such that
	\begin{align*}
	\E \sup_{f \in \tilde{\mathcal{F}}_{\delta_{\bm{\ell}}}} \biggabs{\sum_{1\leq i_j\leq \ell_j, 1\leq j\leq k}  \epsilon_{i_1}^{(1)}\cdots\epsilon_{i_k}^{(k)} \pi_k(f)(X_{i_1}^{(1)},\ldots,X_{i_k}^{(k)})}\leq a_{\bm{\ell}} \bigg(\prod_{j=1}^k \ell_k\bigg)^{1/2} ,
	\end{align*}
	as $\ell_1\wedge\ldots\wedge \ell_k \to \infty$.

	This can be proved following the strategy of that in Theorem \ref{thm:multiplier_CLT}, with a different choice of metric and some resulting technicalities. We provide some details below for the convenience of the reader. Let
	\begin{align*}
	& e_{\bm{\ell},k}^2(f,g)\equiv \frac{1}{\prod_{j=1}^k \ell_j}\sum_{1\leq i_j\leq \ell_j, 1\leq j\leq k} \big(\pi_k(f-g)\big)^2(X_{i_1},\ldots,X_{i_k}),\\
	& \bar{e}_{\bm{\ell},k}^2(f,g)\equiv \frac{1}{\prod_{j=1}^k \ell_j}\sum_{1\leq i_j\leq \ell_j, 1\leq j\leq k} \big(f-g\big)^2(X_{i_1},\ldots,X_{i_k}).
	\end{align*}
	Let $\pnorm{f}{\bm{\ell},k}^2\equiv e_{\bm{\ell},k}^2(f,0)$ and $r_{\bm{\ell},k}(\delta)\equiv \sup_{f \in \tilde{\mathcal{F}}_\delta} \pnorm{f}{\bm{\ell},k}^2$. 	We claim that there exists some $C_0=C_0(k)>0$ such that
	\begin{align}\label{ineq:bootstrap_M_4}
	\log \mathcal{N}\big(\epsilon, \pi_k (\mathcal{F}),\bar{e}_{\bm{\ell},k}\big)\leq  \sum_{r=0}^k \sum_{1\leq j_1<\ldots<j_r\leq k} \log \mathcal{N}\big(\epsilon/C_0, \mathcal{F},e_{(j_1,\ldots,j_r)}\big),
	\end{align}
	where the metric $e_{(j_1,\ldots,j_r)}$ is defined by
	\begin{align*}
	e_{(j_1,\ldots,j_r)}^2(f,g)\equiv \bigg[P^{m-r}\otimes \bigg(\frac{1}{\prod_{q=1}^r \ell_{j_q}} \sum_{1\leq i_{j_q}\leq \ell_{j_q}, 1\leq q\leq r}\delta_{ (X_{i_{j_1}},\ldots,X_{i_{j_r}})}\bigg)\bigg](f-g)^2.
	\end{align*}
	To see this, note that
	\begin{align*}
	\big(\pi_k(f-g)\big)^2 (x_1,\ldots,x_k)  \leq \sum_{r=0}^k \sum_{1\leq i_1<\ldots<i_r\leq k} d_{k,r} \big[ P^{m-r} (f-g)^2(x_{i_1},\ldots,x_{i_r})\big]
	\end{align*}
	holds for some constants $\{d_{k,r}: 0\leq r\leq k\}$, and hence for some $d_k>0$,
	\begin{align*}
	e_{\bm{\ell},k}^2(f,g)&\leq d_k\sum_{r=0}^k \sum_{1\leq j_1<\ldots<j_r\leq k} \frac{1}{\prod_{j=1}^k \ell_j}\sum_{1\leq i_j\leq \ell_j, 1\leq j\leq k}  P^{m-r} (f-g)^2(X_{i_{j_1}},\ldots,X_{i_{j_r}})\\
	&= d_k \sum_{r=0}^k \sum_{1\leq j_1<\ldots<j_r\leq k}  \frac{1}{\prod_{q=1}^r \ell_{j_q}} \sum_{1\leq i_{j_q}\leq \ell_{j_q}, 1\leq q\leq r} P^{m-r} (f-g)^2(X_{i_{j_1}},\ldots,X_{i_{j_r}})\\
	& = d_k \sum_{r=0}^k \sum_{1\leq j_1<\ldots<j_r\leq k} e_{(j_1,\ldots,j_r)}^2 (f,g).
	\end{align*}
	Let $\bar{d_k}\equiv d_k \sum_{r=0}^k \sum_{1\leq j_1<\ldots<j_r\leq k} 1$.  Then
	\begin{align*}
	\log \mathcal{N}\big(\epsilon, \pi_k (\mathcal{F}),\bar{e}_{\bm{\ell},k}\big)\leq  \sum_{r=0}^k \sum_{1\leq j_1<\ldots<j_r\leq k} \log \mathcal{N}\big(\epsilon/\sqrt{\bar{d}_k}, \mathcal{F},e_{(j_1,\ldots,j_r)}\big),
	\end{align*}
	proving the claim (\ref{ineq:bootstrap_M_4}).

	Let $\pnorm{f}{(j_1,\ldots,j_r)}^2\equiv e_{(j_1,\ldots,j_r)}^2(f,0)$. By a conditioning argument and the entropy maximal inequality for Rademacher chaos process (cf. \cite[Corollary 5.1.8 ]{de2012decoupling}), similar to (\ref{ineq:multiplier_clt_1}) we have, with the notation $\mathcal{F}_\delta \equiv \{f_\theta-f_{\theta_0}: \pnorm{\theta-\theta_0}{}\leq \delta\}$ (and $\tilde{\mathcal{F}}$ similarly defined as in the beginning of the proof of Theorem \ref{thm:multiplier_CLT})
	\begin{align}\label{ineq:bootstrap_M_3}
	&\E_\epsilon \biggpnorm{\frac{1}{(\prod_{j=1}^k \ell_j)^{1/2}}\sum_{1\leq i_j\leq \ell_j, 1\leq j\leq k} \epsilon_{i_1}^{(1)}\cdots\epsilon_{i_k}^{(k)} \pi_k(f)(X_{i_1},\ldots,X_{i_k})   }{\tilde{\mathcal{F}}_{\delta}}\nonumber\\
	& \leq C_1 \int_0^{ r_{\bm{\ell},k}(\delta) } \big( \log\mathcal{N}\big(\epsilon, \pi_k(\mathcal{F}),  \bar{e}_{\bm{\ell},k}\big)\big)^{k/2}\ \d{\epsilon} \nonumber\\
	& \leq C_2 \sum_{r=0}^k \sum_{1\leq j_1<\ldots,j_r\leq k} \int_0^{ C_2 r_{\bm{\ell},k}(\delta) } \big( \log\mathcal{N}\big(\epsilon, \mathcal{F},  e_{(j_1,\ldots,j_r)}\big)\big)^{k/2}\ \d{\epsilon} \nonumber  \\
	& \qquad\qquad \qquad\qquad \qquad\qquad\qquad\qquad(\textrm{using the claim (\ref{ineq:bootstrap_M_4})})\nonumber \\
	& = C_2 \sum_{r=0}^k \sum_{1\leq j_1<\ldots,j_r\leq k}  \pnorm{F}{(j_1,\ldots,j_r)} \nonumber\\
	&\qquad\quad\qquad\times \int_0^{ C_2 r_{\bm{\ell},k}(\delta) /\pnorm{F}{(j_1,\ldots,j_r)} } \big( \log\mathcal{N}\big(\epsilon \pnorm{F}{(j_1,\ldots,j_r)} , \mathcal{F},  e_{(j_1,\ldots,j_r)}\big)\big)^{k/2}\ \d{\epsilon}  \nonumber \\
	& \leq  C_2 \sum_{r=0}^k \sum_{1\leq j_1<\ldots,j_r\leq k}  \pnorm{F}{(j_1,\ldots,j_r)} \nonumber\\
&\qquad\quad\qquad\times \int_0^{ C_2 r_{\bm{\ell},k}(\delta) } \sup_Q \big( \log\mathcal{N}\big(\epsilon \pnorm{F}{L_2(Q)} , \mathcal{F},  L_2(Q)\big)\big)^{k/2}\ \d{\epsilon}.  
	\end{align}
	Here $F$ is an envelope for $\mathcal{F}$ which we assume without loss of generality $F\geq 1$. By Proposition \ref{prop:LLN_unbalanced_Ustat}, $\pnorm{F}{(j_1,\ldots,j_r)} \to_p \pnorm{F}{L_2(P)}$ as $\ell_1\wedge\ldots\wedge \ell_k \to \infty$. So we only need to show that $r_{\bm{\ell},k}(\delta) \to_p 0$ as $\ell_1\wedge\ldots\wedge \ell_k \to \infty$ followed by $\delta \to 0$, which reduces to show that
	\begin{align*}
	\sup_{f \in \tilde{\mathcal{F}}_\delta} \biggabs{\frac{1}{\prod_{j=1}^k \ell_j}\sum_{1\leq i_j\leq \ell_j, 1\leq j\leq k} \big( \pi_k(f)^2(X_{i_1},\ldots,X_{i_k})-P^k \pi_k(f)^2\big)}\to_p 0
	\end{align*}
	as $\ell_1\wedge\ldots\wedge \ell_k \to \infty$. This can be shown using similar arguments in the proof of Theorem \ref{thm:multiplier_CLT} by applying Proposition \ref{prop:LLN_unbalanced_Ustat}. The entropy term involving $\pi_k(\mathcal{F})$ can be handled using (\ref{ineq:bootstrap_M_4}) and the above arguments.
\end{proof}

\section{Proofs for Section \ref{section:complex_sampling_M}}

\subsection{Proof of Theorem \ref{thm:M_estimation_sampling_CLT}}

\begin{proof}[Proof of Theorem \ref{thm:M_estimation_sampling_CLT}]
	The proof follows the idea of the proof in Theorem \ref{thm:bootstrap_M}. For notational simplicity, let
	\begin{align*}
	D_{N}^{\pi}(f)&\equiv \frac{1}{m!\binom{N}{m}}\sum_{i_1\neq \ldots \neq i_m}\frac{\xi_{i_1}}{\pi_{i_1}}\cdots \frac{\xi_{i_m}}{\pi_{i_m}} f(X_{i_1},\ldots,X_{i_m}),
	\end{align*}
	and $D(f)=P^m f$ as usual. Further let $\eta_i\equiv \xi_i/\pi_i$ and $g_\theta\equiv f_{\theta}-f_{\theta_0}$. We claim the following:
	\begin{enumerate}
		\item[(Claim 1)] $\sqrt{N}\pnorm{\hat{\theta}_N^{\pi}-\theta_0}{}=\mathcal{O}_{\mathbf{P}}(1)$
		\item[(Claim 2)] For any $\delta_N \to 0$, 
		\begin{align*}
		&\max_{2\leq k\leq m}\sup_{\theta: \pnorm{\theta-\theta_0}{}\leq \delta_N} \biggabs{N^{-k+1}\sum_{i_1\neq \ldots\neq i_k} \eta_{i_1}\cdots\eta_{i_k} \pi_k(f_{\theta}-f_{\theta_0})(X_{i_1},\ldots,X_{i_k}) }=\mathfrak{o}_{\mathbf{P}}(1),
		\end{align*}
	\end{enumerate}
	Then using similar arguments as in the proof of Theorem \ref{thm:bootstrap_M}, we have
	\begin{align*}
	\sqrt{N}\big(\hat{\theta}_N^{\pi}-\theta_0\big) = mV^{-1}\cdot \frac{1}{\sqrt{N}}\sum_{i=1}^N \eta_i \Delta(X_i)+\mathfrak{o}_{\mathbf{P}}(1),
	\end{align*}
	as desired, modulo Claims 1 and 2. Claim 1 can be proved along exactly the same lines as that in the proof of Theorem \ref{thm:bootstrap_M}. Now we prove Claim 2. Note by Proposition \ref{prop:multiplier_ineq_sampling} and using the same notation as in the proof of Theorem \ref{thm:bootstrap_M} 
	\begin{align*}
	&\E \sup_{f \in \tilde{\mathcal{F}}_{\delta_N}} \biggabs{\sum_{i_1,\ldots, i_k} \eta_{i_1}\cdots\eta_{i_k} \pi_k(f)(X_{i_1},\ldots,X_{i_k}) }\\
	&\leq \E \max_{1\leq \ell_1,\ldots,\ell_k\leq N}\sup_{f \in \tilde{\mathcal{F}}_{\delta_N}}\biggabs{\sum_{i_j\leq \ell_j: j=1,\ldots,k} \ \pi_k(f)(X_{i_1},\ldots,X_{i_k}) }\\
	&\leq C \cdot \E \max_{1\leq \ell_1,\ldots,\ell_k\leq N}\sup_{f \in \tilde{\mathcal{F}}_{\delta_N}}\biggabs{\sum_{i_j\leq \ell_j: j=1,\ldots,k} \ \pi_k(f)(X_{i_1}^{(1)},\ldots,X_{i_k}^{(k)}) }\\
	&\qquad\qquad\qquad \textrm{(using the same proofs as in  \cite[Theorem 3.1.1]{de2012decoupling}) }\\
	& = C \cdot \E \bigg[\E_{X^{(1)}} \max_{1\leq \ell_1,\ldots,\ell_k\leq N}\sup_{f \in \tilde{\mathcal{F}}_{\delta_N}}\biggabs{\sum_{i_1=1}^{\ell_1}\sum_{i_j\leq \ell_j: j=2,\ldots,k} \ \pi_k(f)(X_{i_1}^{(1)},\ldots,X_{i_k}^{(k)}) }\bigg]\\
	& = C \cdot \E \bigg[\E_{X^{(1)}} \max_{1\leq \ell_1,\ldots,\ell_k\leq N}\sup_{f \in \tilde{\mathcal{F}}_{\delta_N}}\bigg\lvert\sum_{i_1=1}^{\ell_1}\sum_{i_j\leq \ell_j: j=2,\ldots,k} \ \pi_k(f)(X_{i_1}^{(1)},\ldots,X_{i_k}^{(k)}) \\
	&\qquad \qquad\qquad\qquad\qquad\qquad +\sum_{i_1=\ell_1+1}^{N} \E_{X^{(1)}}\pi_k(f)(X_{i_1}^{(1)},\ldots,X_{i_k}^{(k)})\bigg\lvert \bigg]\\
	&\qquad\qquad\qquad\qquad\qquad\qquad\qquad\qquad\qquad \textrm{(by degeneracy of }\pi_k(f) )\\
	&\leq C \cdot \E \max_{1\leq \ell_2,\ldots,\ell_k\leq N}\sup_{f \in \tilde{\mathcal{F}}_{\delta_N}}\biggabs{\sum_{\substack{1\leq i_1\leq N,\\ i_j\leq \ell_j: j=2,\ldots,k} } \ \pi_k(f)(X_{i_1}^{(1)},\ldots,X_{i_k}^{(k)}) }\\
	&\qquad\qquad\qquad\qquad\qquad\qquad\qquad\qquad\qquad \textrm{(by Jensen's inequality) }\\
	&\leq \cdots\\
	&\leq C \cdot \E\sup_{f \in \tilde{\mathcal{F}}_{\delta_N}}\biggabs{\sum_{1\leq i_1,\ldots,i_k\leq N} \ \pi_k(f)(X_{i_1}^{(1)},\ldots,X_{i_k}^{(k)}) }.
	\end{align*}
	From here the proof of Claim 2 proceeds along the same lines as in the proof of Theorem \ref{thm:bootstrap_M}.
\end{proof}

\subsection{Proof of Theorem \ref{thm:M_estimation_sampling_excess_risk}}

\begin{proof}[Proof of Theorem \ref{thm:M_estimation_sampling_excess_risk}]
	Let $\bar{f}(x)\equiv \E f(x,X_2,\ldots,X_m)$. Note that with the usual notation $\eta_i\equiv \xi_i/\pi_i$, and by similar arguments in the proof of Theorem \ref{thm:M_estimation_sampling_CLT}, 
	\begin{align*}
	&\frac{1}{m!\binom{N}{m}} \sum_{i_1\neq \ldots\neq i_m} \eta_{i_1}\cdots\eta_{i_m} f(X_{i_1},\ldots,X_{i_m})\\
	& =\frac{1}{m!\binom{N}{m}} \sum_{i_1\neq \ldots\neq i_m} \eta_{i_1}\cdots\eta_{i_m} \sum_{j=1}^m \bar{f}(X_{i_j})\\
	&\qquad\qquad + \frac{1}{m!\binom{N}{m}} \sum_{i_1\neq \ldots\neq i_m} \eta_{i_1}\cdots\eta_{i_m}  \sum_{2\leq k\leq m}\sum_{j_1<\ldots<j_k} \pi_k(f)(X_{i_{j_1}},\cdots,X_{i_{j_k}})\\
	& = (1+\mathfrak{o}(1)) m \Prob_N^\pi \bar{f} + C_1 \Delta_N
	\end{align*}
	where
	\begin{align*}
	\Delta_N\leq E_N \equiv  \max_{2\leq k\leq m} N^{-k} \sup_{f \in \mathcal{F}} \biggabs{\sum_{i_1\neq \ldots \neq i_k}\eta_{i_1}\cdots\eta_{i_k} \pi_k(f)(X_{i_1},\ldots,X_{i_k})}.
	\end{align*}
	In other words, $\hat{f}_N^\pi$ is a $C_2 \Delta_N$-empirical risk minimizer of $\arg\min_{\bar{f} \in \bar{\mathcal{F}} } \Prob_N^\pi \bar{f}$. The key observation here is that $\mathcal{E}_P(f)=\mathcal{E}_P(\bar{f})$. On the other hand, it is shown in \cite[Theorem 4.1]{han2018complex} that
	\begin{align*}
	&\Prob\bigg(\sup_{\bar{f} \in \bar{\mathcal{F}}: \mathcal{E}_P(\bar{f})\geq r_N^2} \biggabs{\frac{\mathcal{E}_{\Prob_N^\pi}(\bar{f})}{\mathcal{E}_P(\bar{f})}-1}\geq 3/4\bigg)\\
	&\qquad\leq \frac{C_3}{s}e^{-s/C_3}+\Prob\bigg(\biggabs{\frac{1}{\sqrt{N}}\sum_{i=1}^N\bigg(\frac{\xi_i}{\pi_i}-1\bigg) }>t\bigg)\nonumber
	\end{align*}
	Here the constants $\{C_i\}$ only depend on $\pi_0,\kappa$. Hence \begin{align*}
	& \Prob\big(\big\{\mathcal{E}_P(\hat{f}_N^\pi)\geq r_N^2\}\cap \{\Delta_N\leq r_N^2/(4C_2) \}\big)\\
	&\leq \Prob\big(\big\{\mathcal{E}_P(\hat{f}_N^\pi)\geq r_N^2\}\cap \{\mathcal{E}_{\Prob_N^\pi}(\hat{f}_N^\pi)\leq r_N^2/4 \}\big)\\
	&\leq \Prob\bigg(\sup_{\bar{f} \in \bar{\mathcal{F}}: \mathcal{E}_P(\bar{f})\geq r_N^2} \biggabs{\frac{\mathcal{E}_{\Prob_N^\pi}(\bar{f})}{\mathcal{E}_P(\bar{f})}-1}\geq 3/4\bigg)\\
	&\leq \frac{C_3}{s}e^{-s/C_3}+\Prob\bigg(\biggabs{\frac{1}{\sqrt{N}}\sum_{i=1}^N\bigg(\frac{\xi_i}{\pi_i}-1\bigg) }>t\bigg).
	\end{align*}
	This implies that
	\begin{align*}
	& \Prob\big(\mathcal{E}_P(\hat{f}_N^\pi)\geq r_N^2\big)\\
	&\leq \frac{C_3}{s}e^{-s/C_3}+\Prob\bigg(\biggabs{\frac{1}{\sqrt{N}}\sum_{i=1}^N\bigg(\frac{\xi_i}{\pi_i}-1\bigg) }>t\bigg)+ \Prob\big(\Delta_N>r_N^2/(4C_2)\big)\\
	&\leq \frac{C_3}{s}e^{-s/C_3}+\Prob\bigg(\biggabs{\frac{1}{\sqrt{N}}\sum_{i=1}^N\bigg(\frac{\xi_i}{\pi_i}-1\bigg) }>t\bigg)+ \Prob\big(\Delta_N>u/N\big),
	\end{align*}
	as long as $K_1\geq 2\sqrt{C_2}$. To handle the third probability in the last display, note that for any $p\geq 1$, by Proposition \ref{prop:multiplier_ineq_sampling} and \cite[Corollary 5.1.8 ]{de2012decoupling}
	\begin{align*}
	\E \Delta_N^p \leq \E E_N^p&\leq C_5^p \max_{2\leq k\leq m} N^{-kp} \E \sup_{f \in \mathcal{F}} \biggabs{\sum_{i_1\neq \ldots \neq i_k} \pi_k(f)(X_{i_1},\ldots,X_{i_k})}^p\\
	&\leq C_6^p \max_{2\leq k\leq m} N^{-kp/2} p^{p/(2/m)}=C_6^p N^{-p} p^{p/(2/m)}.
	\end{align*}
	This means that $
	\Prob\big(N\Delta_N>u\big)\leq C_7e^{-u^{2/m}/C_7}$. 
	Combining the estimates in the above displays proves the claim of the theorem.
\end{proof}

\section{Auxiliary results}\label{section:proof_auxiliary_results}

\begin{proposition}\label{prop:LLN_unbalanced_Ustat}
Let $\{X_i\}$ be i.i.d. random variables with law $P$. Let $\mathcal{H}$ be a class of measurable real-valued functions defined on $(\mathcal{X}^m,\mathcal{A}^m)$ with an $P^m$-integrable envelope such that the following holds: for any fixed $\delta>0,M>0, 1\leq k\leq m$, 
\begin{align}\label{cond:LLN_unbalanced_Ustat}
\max_{1\leq j'\leq k} \E  \bigg( \frac{\log \mathcal{N}(\delta, (\pi_k \mathcal{H})_M, e_{\bm{\ell},j'})}{\ell_{j'}} \bigg)^{1/2}\to 0 
\end{align}
holds for any $\ell_1\wedge \cdots \wedge \ell_k \to \infty$. Here for $\bm{\ell}=(\ell_1,\ldots,\ell_k)$ and $\{X_i\}_{i=1}^\infty$, 
\begin{align*}
e_{\bm{\ell},j'}(f,g)\equiv  \frac{1}{\ell_{j'}} \sum_{i_{j'}=1}^{\ell_{j'}} \biggabs{ \frac{1}{\prod_{j\neq j'}{\ell_j}} \sum_{1\leq i_j\leq \ell_j: j\neq j'} (f-g)(X_{i_1},\ldots,X_{i_k})},
\end{align*}
and $(\pi_k \mathcal{H})_M \equiv \{h \bm{1}_{H_k \leq M}: h \in \pi_k \mathcal{H}\}$, where $H_k$ is an envelope for $\pi_k\mathcal{H}$. Then
\begin{align*}
\sup_{h \in \mathcal{H}}\biggabs{\frac{1}{\prod_{k=1}^m \ell_k}\sum_{1\leq i_k\leq \ell_k, 1\leq k\leq m} \big(h(X_{i_1},\ldots,X_{i_m})-P^m h\big)}\to 0
\end{align*}
in $L_1$ as $\ell_1\wedge\ldots\wedge \ell_m \to \infty$. The above display can be replaced by the decoupled version.
\end{proposition}

\begin{remark}\label{rmk:relate_e_j_to_l2}
Note that for any $1\leq j'\leq k$,
\begin{align*}
e_{\bm{\ell},j'}(f,g)&= \frac{1}{\ell_{j'}} \sum_{i_{j'}=1}^{\ell_{j'}} \biggabs{ \frac{1}{\prod_{j\neq j'}{\ell_j}} \sum_{1\leq i_j\leq \ell_j: j\neq j'} (f-g)(X_{i_1},\ldots,X_{i_k})}\\
&\leq \frac{1}{\prod_{j=1}^k \ell_j } \sum_{1\leq i_j\leq \ell_j, 1\leq j\leq k} \abs{f-g}(X_{i_1},\ldots,X_{i_k})\\
&\leq \bigg(\frac{1}{\prod_{j=1}^k \ell_j } \sum_{1\leq i_j\leq \ell_j, 1\leq j\leq k} (f-g)^2(X_{i_1},\ldots,X_{i_k})\bigg)^{1/2},
\end{align*}
so we may use $\ell_2$-type metrics to verify the condition (\ref{cond:LLN_unbalanced_Ustat}).
\end{remark}

\begin{proof}[Proof of Proposition \ref{prop:LLN_unbalanced_Ustat}]
Without loss of generality we assume that $\mathcal{H}$ is $P^m$-centered. By decoupling inequality (more precisely, the proof of \cite[Theorem 3.1.1]{de2012decoupling}), we only need to show
\begin{align}\label{ineq:LLN_unbalanced_Ustat}
\E \sup_{h \in \mathcal{H}}\biggabs{\frac{1}{\prod_{k=1}^m \ell_k}\sum_{1\leq i_k\leq \ell_k, 1\leq k\leq m} h(X_{i_1}^{(1)},\ldots,X_{i_m}^{(m)})}\to 0.
\end{align} 
Note that by expanding $h(x_{i_1},\ldots,x_{i_m}) = (\delta_{x_{i_1}}-P+P)\times(\delta_{x_{i_m}}-P+P) h$, we have that
\begin{align*}
&\biggabs{ \frac{1}{\prod_{k=1}^m \ell_k} \sum_{1\leq i_k\leq \ell_k, 1\leq k\leq m} h(X_{i_1}^{(1)},\ldots,X_{i_m}^{(m)})}\nonumber\\
&\leq C_m  \biggabs{\sum_{k=1}^m \sum_{\sigma_k}  \frac{1}{\prod_{j=1}^k \ell_{\sigma_k(j)}} \sum_{1\leq i_{\sigma_k(j)}\leq \ell_{\sigma_k(j)}: 1\leq j\leq k} (\pi_k h)\big(X_{i_{\sigma_k(1)} }^{ (\sigma_k(1))},\ldots,X_{i_{\sigma_k(k)}}^{(\sigma_k(k))}\big)},
\end{align*}
where the summation over $\sigma_k$ runs over all possible selections of subsets of $\{1,\ldots,m\}$ with cardinality $k$. Since $\pi_k h$ is degenerate of order $k-1$, it follows by a simple conditioning argument that
\begin{align*}
&\E\sup_{h \in \mathcal{H}} \biggabs{\sum_{1\leq i_{\sigma_k(j)}\leq \ell_{\sigma_k(j)}: 1\leq j\leq k} (\pi_k h)\big(X_{i_{\sigma_k(1)} }^{(\sigma_k(1))},\ldots,X_{i_{\sigma_k(k)}}^{(\sigma_k(k))}\big)  } \\
& = \E \bigg\{\E_{X^{\sigma_k(1)}} \sup_{h \in \mathcal{H}}\biggabs{\sum_{1\leq i_{\sigma_k(1)}\leq \ell_{\sigma_k(1)}}   \bigg[ \sum_{ \substack{1\leq i_{\sigma_k(j)}\leq \ell_{\sigma_k(j)}, \\2\leq j\leq k}} (\pi_k h)\big(X_{i_{\sigma_k(1)} }^{(\sigma_k(1))},\ldots,X_{i_{\sigma_k(k)}}^{(\sigma_k(k))}\big) \bigg] } \bigg\}\\
&\lesssim \E \bigg\{\E_{X^{\sigma_k(1)}} \sup_{h \in \mathcal{H}}\biggabs{\sum_{1\leq i_{\sigma_k(1)}\leq \ell_{\sigma_k(1)}}   \bigg[ \sum_{\substack{1\leq i_{\sigma_k(j)}\leq \ell_{\sigma_k(j)},\\ 2\leq j\leq k}  } \epsilon_{i_{\sigma_k(1)} }^{ (\sigma_k(1))}(\pi_k h)\big(X_{i_{\sigma_k(1)} }^{(\sigma_k(1))},\ldots,X_{i_{\sigma_k(k)}}^{(\sigma_k(k))}\big) \bigg] } \bigg\}\\
&\qquad\qquad\qquad\qquad \textrm{(by symmetrization for empirical processes)}\\
& = \E\sup_{h \in \mathcal{H}} \biggabs{\sum_{\substack{1\leq i_{\sigma_k(j)}\leq \ell_{\sigma_k(j)},\\ 1\leq j\leq k } } \epsilon_{i_{\sigma_k(1)} }^{ (\sigma_k(1)) }(\pi_k h)\big(X_{i_{\sigma_k(1)} }^{ (\sigma_k(1))},\ldots,X_{i_{\sigma_k(k)}}^{(\sigma_k(k))}\big) }\\
&= \E\sup_{h \in \mathcal{H}} \biggabs{\sum_{\substack{1\leq i_{\sigma_k(j)}\leq \ell_{\sigma_k(j)},\\ 1\leq j\leq k  } } \bigg(\epsilon_{i_{\sigma_k(1)} }^{(\sigma_k(1))}+\E \epsilon_{i_{\sigma_k(2)} }^{(\sigma_k(2))}+\cdots+ \E \epsilon_{i_{\sigma_k(k)} }^{(\sigma_k(k))} \bigg)(\pi_k h)\big(X_{i_{\sigma_k(1)} }^{(\sigma_k(1))},\ldots,X_{i_{\sigma_k(k)}}^{(\sigma_k(k))}\big) }\\
&\leq  \E\sup_{h \in \mathcal{H}} \biggabs{\sum_{\substack{1\leq i_{\sigma_k(j)}\leq \ell_{\sigma_k(j)},\\ 1\leq j\leq k}} \bigg(\epsilon_{i_{\sigma_k(1)} }^{(\sigma_k(1))}+ \epsilon_{i_{\sigma_k(2)} }^{(\sigma_k(2))}+\cdots+ \epsilon_{i_{\sigma_k(k)} }^{(\sigma_k(k))} \bigg)(\pi_k h)\big(X_{i_{\sigma_k(1)} }^{(\sigma_k(1))},\ldots,X_{i_{\sigma_k(k)}}^{(\sigma_k(k))}\big) }\\
&\qquad\qquad\qquad\qquad \textrm{(by Jensen and independence of decoupled Rademachers)}\\
&\lesssim \E\sup_{h \in \mathcal{H}} \biggabs{\sum_{\substack{1\leq i_{\sigma_k(j)}\leq \ell_{\sigma_k(j)},\\ 1\leq j\leq k}} \bigg(\epsilon_{i_{\sigma_k(1)} }+ \epsilon_{i_{\sigma_k(2)} }+\cdots+ \epsilon_{i_{\sigma_k(k)} } \bigg)(\pi_k h)\big(X_{i_{\sigma_k(1)} },\ldots,X_{i_{\sigma_k(k)}}\big) }\\
&\qquad\qquad\qquad\qquad \textrm{(by undecoupling inequality \cite[Theorem 3.1.2]{de2012decoupling})}\\
&\leq \sum_{j'=1}^k \E\sup_{h \in \mathcal{H}} \biggabs{\sum_{1\leq i_{\sigma_k(j)}\leq \ell_{\sigma_k(j)}, 1\leq j\leq k} \epsilon_{i_{\sigma_k(j')} }(\pi_k h)\big(X_{i_{\sigma_k(1)} },\ldots,X_{i_{\sigma_k(k)}}\big) }.
\end{align*}
Combining the above displays yields that
\begin{align*}
& \E \sup_{h \in \mathcal{H}}\biggabs{\frac{1}{\prod_{k=1}^m \ell_k}\sum_{1\leq i_k\leq \ell_k, 1\leq k\leq m} h(X_{i_1}^{(1)},\ldots,X_{i_m}^{(m)})}\\
&\leq C_m \sum_{k=1}^m \sum_{\sigma_k} \sum_{1\leq j'\leq k} \E\sup_{h \in \mathcal{H}} \biggabs{\frac{1}{\prod_{j=1}^k \ell_{\sigma_k(j)}}\sum_{\substack{1\leq i_{\sigma_k(j)}\leq \ell_{\sigma_k(j)},\\ 1\leq j\leq k}} \epsilon_{i_{\sigma_k(j')} }(\pi_k h)\big(X_{i_{\sigma_k(1)} },\ldots,X_{i_{\sigma_k(k)}}\big) }.
\end{align*}
Hence for (\ref{ineq:LLN_unbalanced_Ustat}) to hold it suffices to prove the following claim: 
\begin{align}\label{ineq:LLN_unbalanced_Ustat_1}
\max_{1\leq j'\leq k}\E \sup_{h \in \mathcal{H}}\biggabs{\frac{1}{\prod_{j=1}^k \ell_j}\sum_{1\leq i_j\leq \ell_j, 1\leq j\leq k} \epsilon_{i_{j'}} (\pi_k h)(X_{i_1},\ldots,X_{i_k})}\to 0
\end{align}
holds for any $1\leq k\leq m$ and $\ell_1\wedge \cdots\wedge \ell_k\to \infty$. Recall that $H^{(k)}$ is the envelope for $\mathcal{H}^{(k)}\equiv\pi_k \mathcal{H}$. Then
\begin{align*}
& \max_{1\leq j'\leq k}\E \sup_{h \in \mathcal{H}}\biggabs{\frac{1}{\prod_{j=1}^k \ell_j}\sum_{1\leq i_j\leq \ell_j, 1\leq j\leq k} \epsilon_{i_{j'}} (\pi_k h)(X_{i_1},\ldots,X_{i_k})}\\
&\leq \max_{1\leq j'\leq k}\E \sup_{h \in \mathcal{H}}\biggabs{\frac{1}{\prod_{j=1}^k \ell_j}\sum_{1\leq i_j\leq \ell_j, 1\leq j\leq k} \epsilon_{i_{j'}} (\pi_k h\bm{1}_{H^{(k)}\leq M})(X_{i_1},\ldots,X_{i_k})} \\
&\qquad\qquad\qquad+ P^k H^{(k)}\bm{1}_{H^{(k)}>M}.
\end{align*}
The second term in the above display vanishes as $M \to \infty$ by the integrability of $H^{(k)}$, and hence we only need to show that
\begin{align}\label{ineq:LLN_unbalanced_Ustat_2}
\max_{1\leq j'\leq k}\E \sup_{h \in \mathcal{H}^{(k)}_M}\biggabs{\frac{1}{\prod_{j=1}^k \ell_j}\sum_{1\leq i_j\leq \ell_j, 1\leq j\leq k} \epsilon_{i_{j'}} h(X_{i_1},\ldots,X_{i_k})}\to 0
\end{align}
holds for any $\ell_1\wedge \cdots\wedge \ell_k\to \infty$ followed by $M \to \infty$. To see this, fix $1\leq j'\leq k$ and $\delta>0$, let $\mathcal{H}_{M,\delta}^{(k)}$ be a minimal $\delta$-covering set of $\mathcal{H}_M^{(k)}$ under $e_{\bm{\ell},j'}$. Then
\begin{align*}
& \E_{\epsilon} \sup_{h \in \mathcal{H}^{(k)}_M}\biggabs{\frac{1}{\prod_{j=1}^k \ell_j}\sum_{1\leq i_j\leq \ell_j, 1\leq j\leq k} \epsilon_{i_{j'}} h(X_{i_1},\ldots,X_{i_k})}\\
&\leq \delta+\E_{\epsilon}  \sup_{h \in \mathcal{H}_{M,\delta}^{(k)}}\biggabs{\frac{1}{\prod_{j=1}^k \ell_j}\sum_{1\leq i_j\leq \ell_j, 1\leq j\leq k} \epsilon_{i_{j'}} h(X_{i_1},\ldots,X_{i_k})}\\
&\leq \delta+\frac{C}{\prod_{j=1}^k \ell_j}\sqrt{\log \mathcal{N}(\delta, \mathcal{H}_M^{(k)}, e_{\bm{\ell},j'})}\sup_{h \in \mathcal{H}_{M,\delta}^{(k)}} \bigg[\sum_{i_{j'}=1}^{\ell_{j'}} \bigg(\sum_{1\leq i_j\leq \ell_j: j\neq j'} h(X_{i_1},\ldots,X_{i_k})\bigg)^2\bigg]^{1/2}\\
&\qquad\qquad \textrm{(by subgaussian maximal inequality, cf. \cite[Lemma 2.2.2]{van1996weak})}\\
&\leq \delta + \frac{C}{\prod_{j=1}^k \ell_j}\sqrt{\log \mathcal{N}(\delta, \mathcal{H}_M^{(k)}, e_{\bm{\ell},j'})}\cdot \bigg(M\sqrt{\ell_{j'}} \prod_{j\neq j'} \ell_j\bigg)\\
& \leq \delta + CM \bigg( \frac{\log \mathcal{N}(\delta, \mathcal{H}_M^{(k)}, e_{\bm{\ell},j'})}{\ell_{j'}} \bigg)^{1/2}
\end{align*}
Hence for any $\delta>0$, by the assumption,
\begin{align}\label{ineq:LLN_unbalanced_Ustat_3}
&\max_{1\leq j'\leq k}\E \sup_{h \in \mathcal{H}^{(k)}_M}\biggabs{\frac{1}{\prod_{j=1}^k \ell_j}\sum_{1\leq i_j\leq \ell_j, 1\leq j\leq k} \epsilon_{i_{j'}} h(X_{i_1},\ldots,X_{i_k})}\nonumber\\
&\leq \delta + CM \max_{1\leq j'\leq k} \E  \bigg( \frac{\log \mathcal{N}(\delta, \mathcal{H}_M^{(k)}, e_{\bm{\ell},j'})}{\ell_{j'}} \bigg)^{1/2}\to 0
\end{align}
as $\ell_1\wedge\cdots\wedge \ell_k \to \infty$ followed by $\delta \to 0$, completing the proof.
\end{proof}

\begin{proposition}\label{prop:multiplier_ineq_sampling}
	Suppose Assumption \ref{assumption:sampling_design} holds. Then with $\eta_i \equiv \xi_i/\pi_i$, for any $r\leq m$, and $p\geq 1$,
	\begin{align*}
	&\E \biggpnorm{\sum_{1\leq i_1,\ldots,i_m\leq n} \eta_{i_1}\cdots\eta_{i_r} f(X_{i_1},\ldots,X_{i_m})}{\mathcal{F}}^p\\
	&\leq (1/\pi_0)^{rp} \E \max_{1\leq \ell_1,\ldots,\ell_r\leq n} \biggpnorm{\sum_{\substack{1\leq i_k\leq \ell_k, 1\leq k\leq r,\\ 1\leq i_k\leq n, r+1\leq k\leq m} }  f(X_{i_1},\ldots,X_{i_m})}{\mathcal{F}}^p.
	\end{align*}
\end{proposition}
\begin{proof}
	The proof is essentially a variant of the proof of Theorem \ref{thm:multiplier_ineq} so we only sketch some details here. Let $\eta_{(1)}\geq \ldots\geq \eta_{(n)}$ be the reversed order statistics of $\{\eta_i\}$. By using $\eta_{(i)}=\sum_{\ell\geq i} \big(\eta_{(\ell)}-\eta_{(\ell+1)}\big)$, we have
	\begin{align*}
	&\E \biggpnorm{\sum_{1\leq i_1,\ldots,i_m\leq n} \eta_{i_1}\cdots\eta_{i_r} f(X_{i_1},\ldots,X_{i_m})}{\mathcal{F}}^p\\
	&=\E \bigg\lVert\sum_{1\leq i_1,\ldots,i_m\leq n} \sum_{\ell_k\geq i_k, 1\leq k\leq r} (\eta_{(\ell_1)}-\eta_{(\ell_1+1)})\cdots(\eta_{(\ell_r)}-\eta_{(\ell_r+1)}) \nonumber\\
	&\qquad\qquad\qquad\qquad\qquad\qquad \times f(X_{i_1},\ldots,X_{i_m}) \bigg\lVert_{\mathcal{F}}^p\nonumber\\
	&\leq \E \bigg[ \bigg(\sum_{1\leq \ell_1,\ldots,\ell_r\leq n} (\eta_{(\ell_1)}-\eta_{(\ell_1+1)})\cdots(\eta_{(\ell_r)}-\eta_{(\ell_r+1)})\bigg)^p\nonumber\\
	&\qquad\qquad\qquad\qquad \times \max_{1\leq \ell_1,\ldots\ell_r\leq n} \biggpnorm{\sum_{\substack{1\leq i_k\leq \ell_k, 1\leq k\leq r,\\ 1\leq i_k\leq n, r+1\leq k\leq m} }  f(X_{i_1},\ldots,X_{i_m})}{\mathcal{F}}^p \bigg]\nonumber\\
	& \leq (1/\pi_0)^{pr} \E \max_{1\leq \ell_1,\ldots,\ell_r\leq n} \biggpnorm{\sum_{\substack{1\leq i_k\leq \ell_k, 1\leq k\leq r,\\ 1\leq i_k\leq n, r+1\leq k\leq m} }  f(X_{i_1},\ldots,X_{i_m})}{\mathcal{F}}^p,
	\end{align*}
	as desired.
\end{proof}

Below we collect some technical lemmas that will be useful in the proofs.

\begin{lemma}[Lemma in \cite{huskova1993consistency}]\label{lem:moment_perm_prod_weight}
	Let $(\xi_1,\ldots,\xi_n)$ be a non-negative vector such that $\sum_{i=1}^n \xi_i = n$. Let $R=(R_1,\ldots,R_n)$ be a random permutation of $\{1,\ldots,n\}$. Then for any $l \in \N$ and $\bm{\alpha}=(\alpha_1,\ldots,\alpha_l) \in \N^l$,
	\begin{align*}
	\biggabs{\E_R \bigg[\prod_{i=1}^l \big(\xi_{R_i}-1\big)^{\alpha_i}\bigg]} \leq C_{l,\bm{\alpha}} n^{-l} \bigg[\sum_{i=1}^n \big(\xi_i-1\big)^2\bigg]^{\sum_i \alpha_i /2}.
	\end{align*}
\end{lemma}

The following result is taken from \cite[Lemma 3.6.15]{van1996weak}.
\begin{lemma}\label{lem:exchangeable_CLT}
	Let $(a_i,\ldots,a_n)$ be a vector and $(\xi_1,\ldots,\xi_n)$ be a vector of exchangeable random variables. Suppose that
	\begin{align*}
	\bar{a}_n =\frac{1}{n}\sum_{i=1}^n a_i= 0,\quad \frac{1}{n}\sum_{i=1}^n a_i^2 \to \sigma^2,\quad \lim_{M\to \infty}\limsup_{n\to \infty} \frac{1}{n}\sum_{i=1}^n a_i^2\bm{1}_{\abs{a_i}>M}=0,
	\end{align*}
	and
	\begin{align*}
	\bar{\xi}_n=\frac{1}{n}\sum_{i=1}^n \xi_i = 0,\quad \frac{1}{n} \sum_{i=1}^n \xi_i^2 \to_{P_\xi} \tau^2,\quad  \frac{1}{n}\max_{1\leq i\leq n}\xi_i^2 \to_{P_\xi} 0.
	\end{align*}
	Then $
	\frac{1}{\sqrt{n}} \sum_{i=1}^n a_i \xi_i \rightsquigarrow_d \mathcal{N}\big(0,\sigma^2\tau^2\big)$. 
\end{lemma}

The following result is taken from \cite[Lemma 3]{cheng2010bootstrap} or \cite[pp. 53]{wellner1996bootstrapping}.

\begin{lemma}\label{lem:transfer_probability}
	The following statements are valid.
	\begin{enumerate}
		\item $\Delta_n = \mathcal{O}_{\mathbf{P}}(1)$ if and only if $\Delta_n\equiv \mathcal{O}_{P_\xi}(1)$ in $P_X$-probability.
		\item $\Delta_n = \mathfrak{o}_{\mathbf{P}}(1)$ if and only if $\Delta_n\equiv \mathfrak{o}_{P_\xi}(1)$ in $P_X$-probability.
	\end{enumerate}
\end{lemma}

\begin{definition}
	A function class $\mathcal{F}$ is \emph{$\alpha$-full} $(0<\alpha<2)$ if and only if there exists some constant $K_1,K_2>1$ such that both
	\begin{align*}
	\log \mathcal{N}\big(\epsilon\pnorm{F}{L_2(\Prob_n)},\mathcal{F},L_2(\Prob_n)\big)\leq K_1\epsilon^{-\alpha},\qquad a.s.
	\end{align*}
	for all $\epsilon>0, n \in \N$, and
	\begin{align*}
	\log \mathcal{N}\big(\sigma\pnorm{F}{L_2(P)}/K_2, \mathcal{F},L_2(P)\big)\geq K_2^{-1}\sigma^{-\alpha}
	\end{align*}
	hold. Here $\sigma^2\equiv \sup_{f \in \mathcal{F}}Pf^2$, $F$ denotes the envelope function for $\mathcal{F}$, and $\Prob_n$ is the empirical measure for i.i.d. samples $X_1,\ldots,X_n$ with law $P$.
\end{definition}

The following result is taken from \cite[Theorem 3.4]{gine2006concentration}.

\begin{lemma}\label{lem:gine_koltchinskii_matching_bound_ep}
	Suppose that $\mathcal{F}\subset L_\infty(1)$ is $\alpha$-full with $\sigma^2\equiv \sup_{f \in \mathcal{F}}Pf^2$. If $
	n\sigma^2\gtrsim_{\alpha} 1$ and $\sqrt{n} \sigma \left({\pnorm{F}{L_2(P)}}/{\sigma}\right)^{\alpha/2}\gtrsim_{\alpha} 1$,
	then there exists some constant $K>0$ depending only on $\alpha,K_1,K_2$ such that
	\begin{align*}
	K^{-1} \sqrt{n} \sigma \bigg(\frac{\pnorm{F}{L_2(P)}}{\sigma}\bigg)^{\alpha/2}\leq \E \biggpnorm{\sum_{i=1}^n \epsilon_if(X_i)}{\mathcal{F}}\leq K\sqrt{n} \sigma \bigg(\frac{\pnorm{F}{L_2(P)}}{\sigma}\bigg)^{\alpha/2}.
	\end{align*}
\end{lemma}

\section*{Acknowledgements}
The author would like to thank Jon Wellner for encouragement and helpful comments on an earlier version of the paper. He would also like to thank anonymous referees for detailed comments and suggestions that significantly improved the article. 

\bibliographystyle{amsalpha}
\bibliography{mybib}

\end{document}